\documentclass[11pt, a4paper]{amsart}
\usepackage[utf8]{inputenc}
\usepackage[english]{babel}
\usepackage{graphicx} 
\usepackage[pagebackref, colorlinks=true,linkcolor=blue,citecolor=red]{hyperref}
\usepackage{pdfpages}

\usepackage{pgfplots}
\usetikzlibrary{calc,3d,shapes, pgfplots.external, intersections}
\newcommand{\qbinom}[2]{\genfrac{[}{]}{0pt}{}{#1}{#2}_{\!2}}
\usepackage[margin=2.5cm]{geometry}
\usepackage{amsmath,amsfonts,amssymb,mathtools,tikz}
\usepackage{tikz}  
\usetikzlibrary{positioning,arrows.meta}
\usepackage{todonotes}

\newtheorem{thm}{Theorem}[section]
\newtheorem*{theorem*}{Theorem}

\newtheorem{prop}[thm]{Proposition}
\newtheorem{lemma}[thm]{Lemma}
\newtheorem{cor}[thm]{Corollary}
\newtheorem{remark}[thm]{Remark}
\newtheorem{example}[thm]{Example}

\newtheorem{definitionn}[thm]{Definition}
\newtheorem{problem}[thm]{Problem}

\DeclareMathOperator{\MH}{\mathrm{MH}}
\DeclareMathOperator{\MC}{\mathrm{MC}}

\newcommand{\cS}{\mathcal{S}}
\newcommand{\cL}{\mathcal{L}}

\newcommand{\Z}{\mathbb{Z}}
\newcommand{\N}{\mathbb{N}}

\title[A characterization of finite modular  groups via   magnitude homology]{A characterization of finite modular  2-groups via   magnitude homology}
\author{Luigi Caputi and Francesco G. Russo}

\date{\today}

\begin{document}

\begin{abstract}
Finite quasihamiltonian groups -- that is, finite groups in which every pair of subgroups commutes -- admit several characterizations.  
In this work, after introducing the  notion of  sectional contractibility graph, we provide  a new characterization in terms of (a suitable subgroup of) magnitude homology. Our approach especially focuses  on the subgroup lattices of finite groups and on the related notion of modularity. We present a detailed study for finite 2-groups, and, in particular, in the case of extraspecial 2-groups and Pauli groups.
    These  groups, of interest because of their diverse applications in both mathematics and physics, already describe the wealth of information that one can obtain in restricting the study to the case of 2-groups. 

\end{abstract}

\maketitle

\section{Introduction}

Magnitude is a numerical invariant of isometry classes of metric spaces -- see~\cite{ Leinsterbook} for a general overview on the topic. 
Its definition  was motivated by considerations on the Euler characteristic in enriched category theory \cite{MR3084566} -- see also~\cite{MR3054629, Leinsterbook}.
Aside from category theory and metric geometry, magnitude has witnessed interest also in other areas such as  differential geometry~\cite{MR3158044},  minimal surfaces~\cite{MR4520035}, geometric measure theory, and potential theory~\cite{MR3701739}. Graphs, endowed with the shortest path distance, are prominent examples of metric spaces. 
In this context, magnitude was first studied in~\cite{magnitudelei}. Rather than on magnitude itself, in this work we are interested in a combinatorially defined graph invariant that categorifies magnitude, the \emph{magnitude homology} of graphs, as defined by Hepworth and Willerton~\cite{HW}, but see also~\cite{LeinsterShulman}. 
 In fact, magnitude homology provides a powerful tool for studying the combinatorial and geometric properties of graphs -- see, for instance~\cite{arXiv:1809.07240,zbMATH07283643,zbMATH07370433,asao2024girth,zbMATH07953420,zbMATH07860680,ivanov}. 
 
 In this work we focus on a specific family of graphs, as we now describe. To each poset $\mathcal{P}$  we can associate the so-called Hasse graph of $\mathcal{P}$; hence, we can compute the associated magnitude homology groups. This approach to magnitude homology of poset was investigated for example in~\cite{KY,ivanov}. 
 We first consider a  nontrivial finite group~$G$. Then we consider  the poset of nontrivial subgroups \[ \cS(G)=(1,G) =\{\,H\leq G\mid 1<H<G\,\}, \] of $G$, ordered by inclusion.    
 We also consider the subgroups lattice $\cL(G)$ of $G$; that is, the poset of \emph{all} subgroups, including $1$ and $G$. Then, we investigate the properties of the corresponding magnitude homology groups, with a particular attention to the  case of almost extraspecial 2-groups -- see Definition~\ref{defextrasp}. Magnitude homology of posets in fact is strictly related to combinatorial and topological properties of the (intervals of the) poset~\cite{KY,ivanov}. Motivated by this relationship, the main motivation of this work is to provide a connection between algebraic properties of groups -- such as modularity -- and the properties of the magnitude homology of their lattices of subgroups. 
 
 In \cite{iwasawa}, Iwasawa classified finite groups in which every pair  $(H,K) \in \cL(G) \times \cL(G)$ satisfies the condition $HK=KH$, that is, he classified those finite groups in which every two  subgroups commute.    These are the so-called \textit{quasihamiltonian groups} and they are nilpotent (see \cite{HHR, iwasawa, S}).  
 Finite groups in which every subgroup is normal are quasihamiltonian -- these are the so-called \textit{hamiltonian groups}. In particular, abelian groups belong to this family.  The subgroup lattice   $\cL(G)$ is  \emph{modular} if for all  $A,B,C \in \cL(G)$ we have $\langle A,B\cap C\rangle   = \langle A,B\rangle\cap C$
whenever $A $ is a subgroup of $C$~\cite{S}. Generally, we refer to a group~$G$ whose subgroup lattice $\cL(G)$ is modular as  \textit{modular group} (or we just say that $G$ is modular). 
A \textit{section} of a finite group $G$ is an epimorphic image of some subgroup of $G$. Denote by $\mu$ the M\"obius function. Our first main result  is the following characterization of modularity;

\begin{thm}[Sectional Contractibility Criterion] \label{thm:sectional-contractibility-modular} Let $G$ be a finite $2$-group. The following conditions are equivalent: 
\begin{enumerate} 
\item[{\rm (i).}]\label{item:modular} $\cL(G)$ is modular; 
\item[{\rm (ii).}]\label{item:all-sections-noncontractible} For every section $M$ of $G$ and every two distinct subgroups $A,B$ of $M$ of order $2$, the poset  $\cS(\langle A,B\rangle)$ is not contractible;
\item [{\rm (iii).}]\label{item:interval-noncontractible} For every  $N$ normal subgroup in $H$ and $H$ subgroup of $G$ and for every two distinct subgroups $A,B$ of $H$ which cover $N$ in the interval $[N,H]$, the open interval $(N, \langle A, B\rangle)$ is not contractible; 
\item[{\rm (iv).}]\label{item:interval-homology} In {\rm (iii)} above, one has  $\widetilde H_0((N,\langle A, B\rangle);\mathbb Z) \cong\mathbb Z^2$ and  $\mu(N,\langle A, B\rangle)=2.$ \end{enumerate}  \end{thm}

We use Theorem~\ref{thm:sectional-contractibility-modular} to provide a characterization of  modular 2-groups in terms of magnitude homology.
First, we  
introduce  the notion of \emph{sectional contractibility graph} (see Definition~\ref{nicenicegraph}) and, for a group~$G$  we consider the value $a(G) \coloneqq \bigl|\{\,A\leq G\mid |A|=2\,\}\bigr|$. Then, we get the following;

 \begin{thm}[Magnitude-Homological Characterization of Modular 2-Groups] \label{cor:magnitude-characterization-modular} Let $G$ be a finite $2$-group. Then the following conditions are equivalent:  
 \begin{enumerate} 
 \item[{\rm (i)}.] $G$ is modular; 
 \item[{\rm (ii)}.] For every section $M$ of $G$, we have \[ \operatorname{rank}_{\mathbb Z}\MH^{\,1}_{2,2}(G_{\mathcal{L}(M)}) = \frac{a(M)(a(M)-1)}{3} \ ,\]
 where $\MH^{\,1}_{2,2}$ is the bottom supported magnitude homology. \end{enumerate} \end{thm}

The notion of \textit{bottom supported magnitude homology}  in Theorem~\ref{cor:magnitude-characterization-modular} (ii)  is introduced  in Definition \ref{bottomsupported}. The sectional contractibility graph of Definition~\ref{nicenicegraph} is strongly related to the bottom supported magnitude homology and allows us to recognize  finite modular groups in terms of purely combinatorial properties. This is not the case if we look at other combinatorial conditions -- such as planarity. On the other hand,  we provide a criterium to identify planarity for sectional contractibility graphs. As a consequence, we can indirectly study magnitude homological groups by looking at certain elementary abelian sections of the group at hand.  Our planarity result is the following;

\begin{thm}[Planarity of the Sectional Contractibility Graph]\label{planarity}
For every finite \(2\)-group \(G\), the sectional contractibility graph $\Gamma_{\mathrm{ct}}(G)$ is planar
if and only if \[m_2(G)=\max\{r \mid \Z(2)^r \ \mbox{is a subgroup of } \ G\}\leq 2.\]
This is equivalent to require that $G$  has no subgroups isomorphic to $\Z(2)^3$.
In particular, when these conditions hold,
\[\Gamma_{\mathrm{ct}}(G)\simeq K_1\vee(tK_2),
\qquad
t=\frac{|\mathrm{Inv}(G)|-1}{2}
\]
where $\Gamma_{\mathrm{ct}}$ is the wedge graph of point graph $K_1$ and of $t$-times the line graph $K_2$, and $\mathrm{Inv}(G)$ is the set of involutions of $G$.
\end{thm}

We conclude the work with the computation of magnitude homology of Pauli groups and some generalizations which appear in \cite{berk1, BBR, babaru, OR}.  The so-called ``large Pauli groups'' are almost extraspecial 2-groups and have been recently studied  in~\cite{BDDR}.

\begin{thm}[Magnitude Homology of  Almost Extraspecial $2$-Groups]
\label{thm:magnitude-almost-extraspecial}
Let $G=P_r$ be an almost extraspecial $2$-group of order
$2^{2r+2}$ with $r\geq1$ positive integers and  $s_r$ positive integer as per Lemma \ref{prop:number-subgroups-almost-extraspecial}
and  $\nu_\ell(r)$ positive integer as per Lemma \ref{prop:number-elementary-sections}. Then
\[
\MH_{k,\ell}\bigl(G_{\cL(P_r)}\bigr)
\cong
\begin{cases}
\mathbb Z^{\,s_r},
   &(k,\ell)=(0,0),\\[2mm]
\mathbb Z^{
  \,2^{\binom{\ell}{2}}\nu_\ell(r)},
   &k=\ell,\quad 1\leq\ell\leq2r+1,\\[2mm]
0,
   &\text{otherwise}.
\end{cases}
\]
In particular, the magnitude homology of $P_r$ is diagonal and torsion-free.
\end{thm}

After recalling the notions and properties of magnitude homology in Section~\ref{sec:MH}, we review the basic concepts of group theory needed in this work  in Section \ref{s2}. Then, we proceed to prove Theorem \ref{thm:sectional-contractibility-modular}. We introduce the sectional contractibility graph in Section~\ref{s3}. Here are placed the proofs of Theorems \ref{cor:magnitude-characterization-modular} and  \ref{planarity}.
Finally, Section \ref{s4} is devoted to prove our last main result, namely Theorem \ref{thm:magnitude-almost-extraspecial}, and the computations. 

\subsection*{Acknowledgements}
The authors acknowledge support from the grants JSPS Bilateral Collaborations  Grant No. JPJSBP120256504
and 
KAKENHI  (Japanese Society for the Promotion of Science, Tokyo, Japan)
- Grant No. 26K2225. This work was supported also by the ``National
Group for Algebraic and Geometric Structures, and their Applications'' (GNSAGA - INdAM). 

\subsection*{AI usage statement}
ChatGpt Astra was used for image generation and as   computational aid in the computation of the magnitude homology of the Pauli groups.

\section{Magnitude homology}
\label{sec:MH}
In this section we recall  the definition of magnitude homology of graphs and posets. We will mainly follow~\cite{HW, asao}.

First, observe that a connected undirected graph can be seen as a metric space with the path metric -- the distance between two vertices of the graph being given by the length of a shortest path connecting them. 
To be more precise, the points of the metric space associated to a graph~$\Gamma$ are its vertices, and edges are declared to have length~$1$. Concretely, the metric~$d_{\Gamma}$ on $\Gamma$ is  given by
\begin{align*}
    d_{\Gamma}(v,w)\coloneqq \min  \{ d_{\Gamma}(v,v_1) + &\sum_{i=1}^{k-2} d_{\Gamma}(v_i,v_{i+1}) + d_{\Gamma}(v_{k-1},w) \mid \\
    & \{v,v_1\}, \{v_{k-1},w\}, \{v_i,v_{i+1}\}\in E(\Gamma), i=1,\dots, k-2 \}
\end{align*}
for $v,w$ vertices of $\Gamma$. If the graph is not connected, we set $d_{\Gamma}(v,w)\coloneqq \infty$ for every $v,w\in \Gamma$ not connected by any path; hence, for non-connected graphs, $d_{\Gamma}$ is an extended metric. For a $(k+1)$-tuple~$(v_0,\dots,v_k)$ of vertices of $\Gamma$, with $v_i\neq v_{i+1}$ and $d_{\Gamma}(v_i,v_{i+1})<\infty$ for each~$i$, the \emph{length} of $(v_0,\dots,v_k)$ in $\Gamma$ is the number 
\[
\ell(v_0,\dots,v_k)\coloneqq \sum_{i=0}^{k-1}d_{\Gamma}(v_i,v_{i+1}) \ .
\]
We  recall the definition of magnitude chain groups. Let $R$ be a commutative unital ring.

\begin{definitionn}\label{def:magnitude}
    Let $\Gamma$ be a graph, and $l,k\in \mathbb{N}$ natural numbers. We let
    \[
    \MC_{k,l}(\Gamma;R)\coloneqq R\langle (v_0,\dots,v_k) \mid v_0\neq \dots\neq v_k, \ell(v_0,\dots,v_k)=l \rangle 
    \]
    be the free $R$-module on the paths of length $l$ on $k+1$ vertices  of $\Gamma$. 
    The differential
    $
    \partial\colon\MC_{k,l}(\Gamma;R)\to\MC_{k-1,l}(\Gamma;R)
    $
    is defined on $(k+1)$-uples $(v_0,\dots,v_k)$ by
    \[
    \partial(v_0,\dots,v_k)\coloneqq \sum_{i=1}^{k-1} (-1)^i\partial_i(v_0,\dots,v_k) \ ,
    \]
    where $\partial_i(v_0,\dots,v_k)=(v_0,\dots,v_{i-1},v_{i+1}, \dots, v_k)$  assuming that \[\ell(v_0,\dots,v_k)=l=\ell(v_0,\dots,v_{i-1},v_{i+1}, \dots, v_k)\ ,\] and $\partial_i$ is set to be $0$ otherwise. 
\end{definitionn}

For a given $l\in \mathbb{N}$, the pair $(\MC_{*,l}(\Gamma;R), \partial)$ is a chain complex by \cite[Lemma~2.11]{HW}.

\begin{definitionn}[{\cite[Definition~2.4]{HW}}]
The \emph{magnitude homology}~$\MH_{*,*}(\Gamma;R)$ of $\Gamma$ is defined as the bigraded $R$-module $\bigoplus_{k,l\geq 0} \MH_{k,l}(\Gamma;R)$, where
    \[
    \MH_{k,l}(\Gamma;R)\coloneqq H_k(\MC_{*,l}(\Gamma;R), \partial)
    \]
    is given by the homology of the magnitude chain complex.
\end{definitionn}

Magnitude homology is an homology theory of  graphs, and it is functorial with respect to contractions. We recall that a contraction of a graph $\Gamma$ with respect to an edge~$e$ is the graph obtained from $\Gamma$ by contracting $e$ to a point. Then, if $\Gamma$ and $\Gamma'$ are graphs, we consider maps $\phi\colon \Gamma\to \Gamma'$ of vertices that preserve or contract each edge of~$\Gamma$. Observe that such maps do not increase the length of tuples of vertices of $\Gamma$: that is we have
$
\ell(\phi(v_0),\dots,\phi(v_k))\leq \ell(v_0,\dots,v_k) 
$.
Every contraction $\phi\colon \Gamma\to \Gamma'$ of graphs induces a chain map
\[
\phi_\# \colon \MC_{*,*}(\Gamma;R)\to \MC_{*,*}(\Gamma';R) 
\]
which to a tuple $(v_0,\dots,v_k)$ of $\Gamma$ associates the tuple $(\phi(v_0),\dots,\phi(v_k))$ if  $\ell(\phi(v_0),\dots,\phi(v_k))$ equals the length $\ell(v_0,\dots,v_k)$, and it is set to be $0$ otherwise. The map $\phi_\#$ is a chain map, as it commutes with the differential $\delta$, and it induces a map
in magnitude homology. If we denote by $\mathbf{Graph}$ the category of graphs and contractions, we have the following

\begin{prop}[{\cite[Proposition~3.3]{HW}}]
    Magnitude homology is a functor
    \[
    \MH_{*,*}\colon \mathbf{Graph}\to \mathbf{BiGrMod}_R
    \]
    from the contraction category of graphs to the category of bigraded $R$-modules.
\end{prop}

The definition of magnitude homology can be generalized verbatim to arbitrary \emph{directed} graphs, with the observation that a directed graph  can be seen as a metric space with the path metric~$d$, where $d(v,w)$ is the minimum length across all \emph{directed} paths between $v$ and $w$ -- see, e.g.~\cite{asao}.  Magnitude homology is known to satisfy a K\"unneth Theorem with respect to the box product of undirected graphs~\cite[Theorem~5.3]{HW}. For generalized metric spaces, including directed graphs, this was shown in \cite[Theorem~4.6]{roff}.

In the following, diagonal graphs will play a prominent role. Recall that a graph is called \emph{diagonal} if its magnitude homology is concentrated on the diagonal, that is $\MH_{k,l}(G)=0$ for $k\neq l$~\cite[Definition~7.1]{HW}. Diagonality of graphs was shown to be connected to relevant properties of the graphs at hand~\cite{HW, bottinelli2020magnitude, asao2024girth, ivanov}. Furthermore, it is preserved under taking Cartesian products of graphs~\cite[Proposition~7.3]{HW} and any non-trivial join of graphs is diagonal~\cite[Theorem~7.5]{HW}.

\begin{example}[{\cite[Section 7]{HW}}]
\label{ex:treesdiagonal}
    Trees and complete graphs are diagonal.
\end{example}

We point out here that the notion of diagonality for graphs extends to digraphs~\cite{ivanov} in the obvious way.

We now focus on the magnitude homology of posets. We first need to recall some terminology on posets.  
Let $(\mathcal{P},\leq)$ be a finite poset. We call $(\mathcal{P},\leq)$ \emph{ranked} if all maximal chains in $\mathcal{P}$ have the same number of elements. Recall that $y$ covers $x$ in $\mathcal{P}$ if $x< y$ and there is no $z$ in $\mathcal{P}$ with $x< z< y$. If $\mathcal{P}$ is a finite ranked poset there is a uniquely defined rank function $r\colon \mathcal{P}\to\N$ such that:
\begin{enumerate}
    \item $r(x)=0$ if and only if $x$ is a minimal element in $\mathcal{P}$;
    \item if $y$ covers $x$, then $r(y)=r(x)+1$.
\end{enumerate}
Observe that all maximal elements in a ranked poset have the same rank~$d+1$, which we call the \emph{dimension} of $\mathcal{P}$. Let $\mathcal{P}$ be  a ranked poset with rank function $r$, and  $R$ be a ring. We say that $\mathcal{P}$ is \textit{Cohen-Macaulay} over $R$  if for each open interval $(x,y)$ in $\mathcal{P}$ we have 
\[
\widetilde{H}_{i}((x,y);R)=0
\]
for $i\neq r(y)-r(x)-2$. We further recall that if $\mathcal{P}$ is \textit{lexicographically shellable}, then $\mathcal{P}$ is Cohen-Macaulay over $\Z$ and over all fields -- see, e.g.~\cite[Theorem~3.2]{BGS}.  This implies that the top homology in a  Cohen-Macaulay poset is completely described by the M\"obius function. In fact, we have:
\[
\dim_\mathbb{F} \widetilde{H}_{r(y)-r(x)-2}((x,y); \mathbb{F})=(-1)^{r(y)-r(x)}\mu(x,y)\geq 0
\]
where $\mathbb{F}$ is a field and $\mu$ denotes its M\"obius function.

 If $\mathcal{P}$ is a finite ranked poset, we can consider the directed graph $G_\mathcal{P}$ 

\begin{enumerate} 
\item[$\bullet$] whose vertices $V(G_\mathcal{P})$ are the elements of $\mathcal{P}$; 
\item[$\bullet$] whose elements  $x$ and $y$ are joined, i.e.~$(x,y)\in E(G_\mathcal{P})$, if and only if $y$ covers $x$. 
\end{enumerate}
Observe that, if $r$ is the rank function on $\mathcal{P}$, then the path distance $d_{G_\mathcal{P}}(x,y)$ in $G_\mathcal{P}$ is given by
\[
d_{G_\mathcal{P}}(x,y)=r(y)-r(x)\ .
\]
The magnitude homology of the poset $P$ is then set to be the magnitude homology of the directed graph~$G_\mathcal{P}$. We will use the following characterization of magnitude homology of posets:

\begin{thm}[{\cite[Proposition~7.2]{ivanov}}]\label{thm:decompositionMH}
    Let $\mathcal{P}$ be a ranked poset. Then, for $\ell\geq 1$ we have:
    \[
    \MH_{k,\ell}(G_\mathcal{P})\cong \underset{x<y}{\underset{r(y)-r(x)=\ell} \bigoplus}  \widetilde{H}_{k-2}((x,y)) \ ,
    \]
    where $(x,y)$ denotes the open interval in $\mathcal{P}$. 
\end{thm}

As a consequence, if $\mathcal P$ is a ranked Cohen-Macaulay poset, then $G_{\mathcal P}$ is diagonal. We shall use this fact several times.

\section{A sectional contractibility criterion}\label{s2}

The aim of this section is to provide a sectional contractibility criterion for finite $2$-groups. First we recall some notions on the posets of subgroups associated to any finite group, and then we prove the criterion.

\subsection{The posets of subgroups}
For a nontrivial finite group $G$, we shall  consider the poset \[ \cS(G)=(1,G) =\{\,H\leq G\mid 1<H<G\,\} \ , \] of its subgroups, ordered by inclusion. We shall also consider the subgroups lattice $\cL(G)$ of~$G$ (including $1$ and $G$). Observe that $\cL(G)$ contains  $ \cS(G)$ properly. Denote by $\langle X\rangle$ the subgroup generated by $X$. Following \cite{S}, $\cL(G)$ is  \emph{modular} if for all  $A,B,C \in \cL(G)$ we have $\langle A,B\cap C\rangle   = \langle A,B\rangle\cap C$
whenever $A $ is a subgroup of $C$. Generally, we refer to a group $G$ whose subgroup lattice $\cL(G)$ is modular as  \textit{modular group} (or we just say that $G$ is modular).

For the sake of completeness, it is useful to remind here the definition of the \emph{order complex}~$\Delta(\mathcal{P})$ of a poset~$\mathcal{P}$. 
This is the abstract simplicial complex whose vertices are the elements of $\mathcal{P}$ and whose $k$-dimensional simplices are the strictly increasing chains $ x_0<x_1<\cdots <x_k$ in~$\mathcal{P}$. The geometric realization of $\Delta(\mathcal{P})$ will be denoted by  $|\Delta(\mathcal{P})|.$ When discussing the homotopy type (resp.~the homology) of $\mathcal{P}$, we shall mean the homotopy type (resp.~the homology) of~$|\Delta(\mathcal{P})|$. Thus, for instance, if $j$ is a nonnegative integer, then \[ \widetilde H_j(\mathcal{P};R) \coloneqq\widetilde H_j(|\Delta(\mathcal{P})|;R) \] for every coefficient ring $R$.

We are interested in conditions for which $\cS(G)$, or more generally an interval $(H,K)$ in $\cL(G)$ with $H,K$ subgroups of $G$, is contractible. By \cite[Proposition~1.8]{kratzer}, $\cS(G)$ is contractible if $G$ admits a normal subgroup without complement. An example is given by the Frattini subgroup $\mathrm{Frat}(G)$. For $H\leq G$, define $\mathrm{Frat}(H,G)$ to be the intersection of all maximal subgroups of $G$ containing $H$.

\begin{lemma}[{\cite[Lemma 4.1]{kratzer}}]\label{kratzerlemma}
    If $\mathrm{Frat}(H,G)\neq H$, then the interval $(H,G)$ is contractible.
\end{lemma}

As a corollary we have that if the Frattini subgroup is nontrivial, then $\cS(G)$ is contractible.
In most cases instead, the open intervals have the homotopy type of a wedge of spheres;

\begin{remark}\rm For  a finite \(p\)-group $G$ with $H$ proper subgroup, set $n=\log_p  |G:H|$. Then, note that  for any \(n\geq2\) one has
\[|\Delta((H,G))| \simeq \bigvee_{p^{\binom n2}} S^{n-2} \ \ \ \ \  \ \ \ \mbox{if} \ \ \mathrm{Frat} (G)\leq H\ ,\]
while $|\Delta((H,G))|$ is the empy complex if $\mathrm{Frat}(G)\nleq H$. \end{remark}

For every finite group \(G\) and every proper subgroup \(H<G\), one  has
\[\mu(H,G)= \widetilde \chi \bigl(\Delta((H,G))\bigr)\ ,
\]
where $\mu(H,G)$ is the M\"obius function of $\mathcal{S}(G)$
 and $\widetilde \chi$ the reduced Euler characteristic of the associated order complex; see \cite{bjorner80, bjorner81, kratzer, quillen78, Shareshian01}. 

We recall also the following theorem which will be helpful in the follow-up;

\begin{thm}[{\cite[Theorem 3.3]{bjorner80}}]
\label{thm:bjorner}
  Let $G$ be a finite group. Then, the following conditions are equivalent. 
  \begin{enumerate}
      \item[{\rm (i).} ] $G$ is supersolvable;
      \item[{\rm (ii).} ]  $\cL(G)$ is Cohen-Macaulay;
      \item[{\rm (iii).} ]  $\cL(G)$ is lexicographically shellable.
  \end{enumerate}
\end{thm}
 
Let $G$ be a supersolvable group. Suppose  that one has a homotopy equivalence
$|\Delta((H,G))|
\simeq
\bigvee_b S^{n-2}.$  Thus ,  we have
\[
|\Delta((H,G))|
\simeq
\bigvee_{|\mu(H,G)|}S^{n-2}\ .
\]
A wedge of zero spheres is interpreted as a one-point space. Here supersolvability ensures that the subgroup lattice is graded and lexicographically shellable. Shellability passes to intervals and their proper parts.
Details can be found in \cite{quillen78, zbMATH02012754, Shareshian01}.

\subsection{On the characterizations of modularity}
In this section we prove the sectional contractibility criterion for $2$-groups given in Theorem~\ref{thm:sectional-contractibility-modular}. First, we need to introduce some terminology.

\begin{definitionn}[{\cite[Vol.II, Appendix 17]{berk1}}]\label{qstaredstargroups}
Denoting  the finite generalized quaternion $2$-group  for some $m \ge2$ by \[Q_{2^{m+1}}=\langle a,b \mid a^{2^m}=1, b^2=a^{2^{m-1}}, b^{-1}ab=a^{-1}\rangle\ ,\]
 we say that a finite $2$-group $G$ is a $Q^\times$-group if for some $r \ge 0$ 
\[G=Q_{2^{m+1}} \times \underbrace{\mathbb{Z}(2) \oplus \ldots \oplus \mathbb{Z}(2)}_{r-\mathrm{times}}= Q_{2^{m+1}} \times {\mathbb{Z}(2)}^r\ .
\] 
Denoting the finite generalized dihedral $2$-group for some $m \ge 2$  by
\[D_{2^{m+1}}=\langle a,b \mid a^{2^m}=b^2=1, b^{-1}ab=a^{-1}\rangle\ ,
\]
we say that a finite $2$-group $G$ is  a $D^\times$-group if 
for some $s \ge 0$
\[G= D_{2^{m+1}} \times \underbrace{\mathbb{Z}(2) \oplus \ldots \oplus \mathbb{Z}(2)}_{s-\mathrm{times}} = D_{2^{m+1}} \times {\mathbb{Z}(2)}^s\ .
\] 
\end{definitionn}

Definition \ref{qstaredstargroups} gives the usual \emph{quaternion group}  $Q_8$ of order $8$  and the usual \emph{dihedral group}~$D_8$ of order $8$ for $m=2$. Recall that finite groups in which every pair  $(H,K) \in \cL(G) \times \cL(G)$ satisfies the condition $HK=KH$, that is,  those finite groups in which every two  subgroups commute, are called \textit{quasihamiltonian}. Recall also  that finite groups in which every subgroup is normal are called \textit{hamiltonian groups}. Hamiltonian 2-groups are of the form \[ Q_8 \times {\mathbb{Z}(2)}^r\] for some $r \ge 0$ by a result of Dedekind and Hamilton
 \cite[Vol. I, \S 1, Theorem 1.20]{berk1}. This structure is peculiar; in fact, finite hamiltonian 2-groups are $Q^\times$-groups when $m=2$.  

 Recall that a \textit{section} of a finite group $G$ is an epimorphic image of some subgroup of $G$. 

\begin{definitionn}[{\cite[Vol.II, \S 56]{berk1}}]\label{qfree} A finite $2$-group $G$ is \emph{quaternion-free} if it has no sections isomorphic to $Q_8$. A finite $2$-group  is \emph{dihedral-free} if it has no sections isomorphic to $D_8$.
\end{definitionn}

 Iwasawa \cite[Theorem 2.4.14]{S} recognized among all quasihamiltonian groups, not only abelian groups and hamiltonian groups, but also modular groups, see  \cite[\S 73, Theorem 73.15, Vol.II]{berk1} and \cite[Chapter 13]{HHR}. In fact, the theory of Iwasawa successfully applies  to several types of algebraic structures, as illustrated in \cite{HHR, S}.  For finite $2$-groups, we have the following characterization;

\begin{thm}[{\cite[Proposition 79.6, \S 79, Vol. II]{berk1}}]\label{grandeiwasawa}
A finite 2-group $G$ is modular if and only if $G$ is dihedral-free. A finite 2-group $G$ is modular and quaternion-free if and only if $G$ contains an abelian normal subgroup $A$ with cyclic quotient $G/A$ and there exists an element $t \in G$  and an  integer $s \ge 2$  such that $G =\langle A,t \rangle$ and $t^{-1}at=a^{1+2^s}$ for all $a \in A$. 
\end{thm}

In particular, finite quasihamiltonian groups are nilpotent, then supersolvable. This allowed a series of  homotopical considerations to Shareshian and others \cite{BGS, bjorner80, Shareshian01, woodroofe}.

We wish to extend the computations in the last section to modular groups and  quasihamiltonian groups, and to this scope, it's useful to recall the following classical notion on products of groups -- see  \cite{berk1, babaru, BBR, OR, S}.

\begin{definitionn}\label{def:1}
Let $G_1$ and $G_2$ be  two finite groups  with $A_1 $ subgroup which is contained in the center $ Z(G_1)$ of $G_1$, similarly for $A_2$ in $ Z(G_2)$ and 
$\varphi \colon A_1 \rightarrow A_2$ isomorphism. Consider the subgroup $N_{\varphi}  = \{\,(a,\varphi(a)^{-1})\mid a\in A_1\,\}$ in the direct product 
 $G_1\times G_2=\{(g_1,g_2) \mid g_1 \in G_1, g_2 \in G_2\}$ which forms a group with the pointwise multiplication. Since $N_{\varphi}$ lies in $Z(G_1 \times G_2)$ by construction,   \[G_1 \circ G_2 =
\frac{G_1\times G_2}{N_{\varphi}}\]  is called  (external) central product of $G_1$ and $G_2$ (amalgamating
$A_1$ and $A_2$ via $\varphi$). 
\end{definitionn}

Note that there are several equivalent formulations of Definition \ref{def:1}, see  \cite[Vol. I, p.11]{berk1}. We  introduce a specific notation, when we have the dihedral group of order 8 \[D_8=\langle r,s\mid r^4=s^2=1,\ srs=r^{-1}\rangle  \] and  for the quaternion group of order 8 \[Q_8=\langle  i,j \mid i^4=j^4=1, \ (ij)^2=(ji)^{-1}\rangle\ ,\]  since both of them provide important evidences for  Definition \ref{def:1}. For instance, when $n \ge 1$ we write
\[
E^+_{{2}^{2n+1}} = \underbrace{D_8 \circ D_8 \circ \ldots  \circ D_8}_{n-\mathrm{times}} \ \qquad \ \mbox{and} \ \qquad E^-_{{2}^{2n+1}} = \underbrace{(D_8 \circ D_8 \circ \ldots  \circ D_8)}_{(n-1)-\mathrm{times}} \circ \ Q_8.
\]

The \textit{commutator subgroup} of two subgroups $H$ and $K$ of a group $G$ is denoted by $\mathrm{comm}(H,K)$. In particular, $\mathrm{comm}(G,G)=G'$ is the \textit{derived subgroup} of $G$ and it is always a normal subgroup with abelian quotient $G/G'$, see \cite{berk1, HHR, S}. For the following definition of extraspecial,  generalized extraspecial and almost extraspecial $p$-groups, see e.g.~\cite{berk1}.

\begin{definitionn}\label{defextrasp}
A  group $G$ of prime power order $|G|=p^n$  is \textit{extraspecial}, if \[Z(G)= \mathrm{comm}(G,G)=\mathrm{Frat}(G) \cong \mathbb{Z}(p)\ ,\] 
that is, if $Z(G)=\mathrm{comm}(G,G)=\mathrm{Frat}(G)$ is cyclic of order $p$. Instead, we say that $G$  is \textit{almost extraspecial}, if $\mathrm{comm}(G,G)=\mathrm{Frat}(G) \cong \mathbb{Z}(p)$ and $Z(G) \cong \mathbb{Z}(p^2)$. Finally, we say that $G$ is  \textit{generalized extraspecial}, if $\mathrm{comm}(G,G)=\mathrm{Frat}(G) \cong\mathbb{Z}(p)$ is  contained in $Z(G)$.
\end{definitionn}

These groups have been recently studied in \cite{BBR, BDDR, OR} and attracted interest in the literature for various reasons. It is useful to note that any finite 2-group is extraspecial if, and only if, it is either of type $E^+_{{2}^{2n+1}}$, or of type $E^-_{{2}^{2n+1}}$.  More generally, any finite $p$-group is extraspecial if and only if it can be split appropriately as central product of some of its subgroups. Moreover, there are examples which show that there are generalized extraspecial $p$-groups which are not extraspecial and that there are almost extraspecial $p$-groups which are not extraspecial, however both these two generalizations can be appropriately described via central products as well (see \cite{berk1,  OR, S}).

We will come back to these groups later, after recalling the version of the Solomon--Tits Theorem which will be used for subgroup lattices of elementary abelian groups later on. We start by recalling the notion of spherical buildings, see e.g.~\cite[Example 18.1.11]{davis}.

\begin{definitionn}
\label{titsposet} \rm Let $V$ be an $n$-dimensional vector space over a field $\mathbb{F}$, where $n\geq 2$. The \emph{Tits poset} of $V$ is the poset  $\mathcal B(V) =\{\,W\leq V\mid 0<W<V\,\},$  ordered by inclusion. The order complex  $\Delta(\mathcal B(V))$ is called the \emph{spherical building of type $A_{n-1}$ associated with $V$}. Its simplices are the proper nontrivial flags  $0<W_0<W_1<\cdots <W_r<V.$  Its maximal simplices, also called \emph{chambers}, correspond to complete flags  $0<V_1<V_2<\cdots <V_{n-1}<V,$ where $\dim_\mathbb{F} V_j=j.$  Consequently, $\Delta(\mathcal B(V))$ has dimension $n-2$.  \end{definitionn}

From Definition \ref{titsposet}, if  $\mathcal E=(v_1,\ldots,v_n)$ is a basis of $V$, the subcomplex consisting of the flags whose members are spanned by subsets of $\mathcal E$ is called the \emph{apartment associated with $\mathcal E$}. Each such apartment is isomorphic to the barycentric subdivision of the boundary of an $(n-1)$-simplex and is therefore homeomorphic to $S^{n-2}$. The classical Solomon–Tits theorem says that that the order complex of the Tits poset is homotopy equivalent to a wedge of spheres -- see~\cite{solomon, davis}:

\begin{thm}
\label{solomontits} Let $V$ be an $n$-dimensional vector space over a field $\mathbb{F}$, with $n\geq 2$. Then the spherical building of type $A_{n-1}$ associated with $V$ is homotopy equivalent to a wedge of $(n-2)$-dimensional spheres. In other words, there exists an indexing set $I$ such that \[ |\Delta(\mathcal B(V))| \simeq \bigvee_{\alpha\in I}S^{n-2}. \] In particular,  $\widetilde H_j(\mathcal B(V);\mathbb Z)=0$ for every $j\neq n-2,$ and $\widetilde H_{n-2}(\mathcal B(V);\mathbb Z)$ is a free abelian group. \end{thm}

 The Solomon--Tits Theorem can be formulated generally for arbitrary spherical buildings: a spherical building of rank $r$ is homotopy equivalent to a wedge of $(r-1)$-dimensional spheres. The preceding statement is the specialization to the building of type $A_{n-1}$, whose rank is $n-1$.

\begin{cor} \label{titsfact} Let $\mathbb F_q^n$ be the vector space of dimension $n$ over the finite field $\mathbb{F}_q$ of prime power order $q$  and $n\geq 2$. Then \[ |\Delta(\mathcal B(\mathbb F_q^n))| \simeq \bigvee_{q^{\binom n2}}S^{n-2}. \]
\end{cor}

\begin{proof} By Theorem~\ref{solomontits}, $\mathcal B(V)$ is homotopy equivalent to a wedge of spheres of dimension $n-2$. It remains to determine their number. The poset $\mathcal B(V)$ is the open interval $(0,V)$ in the lattice of subspaces of $V$. The M\"obius function of this lattice satisfies \[ \mu(0,V)=(-1)^n q^{\binom n2}\ . \] On the other hand, \[ \widetilde\chi(\mathcal B(V))=\mu(0,V)\ . \] If $\mathcal B(V)$ is homotopy equivalent to a wedge of $m$ spheres of dimension $n-2$, then \[ \widetilde\chi(\mathcal B(V)) =m(-1)^{n-2}=m(-1)^n. \] It follows that \[ m=q^{\binom n2}, \] which proves the result. \end{proof} 

We now apply these facts to the case of elementary abelian $p$-groups. 

\begin{cor} Given  an elementary abelian $p$-group $\mathbb{Z}(p)^n$ of rank $n\geq 2$ with poset of proper nontrivial subgroups  $\cS(\mathbb{Z}(p)^n)$, we have an homotopy equivalence \[ |\Delta(\cS(\mathbb{Z}(p)^n))| \simeq \bigvee_{p^{\binom n2}}S^{n-2}\ . \] 
\end{cor}

\begin{proof} Regarding $\mathbb{Z}(p)^n$ as the additive group of the vector space $\mathbb F_p^n$, the subgroups of $\mathbb{Z}(p)^n$ are precisely its $\mathbb F_p$-linear subspaces. Hence there is an isomorphism of posets $\cS(\mathbb{Z}(p)^n)\cong\mathcal B(\mathbb F_p^n).$ The conclusion follows from Corollary \ref{titsfact}. \end{proof}

\begin{remark} \rm The group  $\widetilde H_{n-2} \bigl(\mathcal B(\mathbb F_q^n);\mathbb Z\bigr)$ is often called the \emph{Steinberg module}, see \cite{davis, solomon}.
\end{remark}

If $G$ is a finite $2$-group, then we know by Theorem~\ref{grandeiwasawa} that $G$ is modular if and only if $G$ is dihedral-free.

\begin{lemma}\label{crucialem}
Let  $G$ be a finite modular $2$-group. Then $\mathcal{S}(G)$ is contractible
if and only if $G$ is not elementary abelian. In fact, the same statement holds for every finite $2$-group.
\end{lemma}

\begin{proof}
Let $G$ be a finite $2$-group. Assume first that $G$ is not elementary abelian. Then $G$ is not generated by elements of order $2$ alone, and therefore its Frattini subgroup $\mathrm{Frat}(G)$ is nontrivial. Since $\mathrm{Frat}(G)$ is the intersection of all maximal subgroups of $G$, it cannot admit a complement in $G$: indeed, if $G=\mathrm{Frat}(G)H$ for some subgroup $H\leq G$, then the defining property of the Frattini subgroup forces $H=G$, so no proper subgroup can complement $\mathrm{Frat}(G)$. As $\mathrm{Frat}(G)$ is a nontrivial normal subgroup of $G$ without complement, \cite[Proposition~1.8]{kratzer} yields that $\cS(G)$ is contractible.

Conversely, suppose that $G$ is elementary abelian, say $G\cong \mathbb{Z}(2)^n$.
Then the proper nontrivial subgroups of $G$ are precisely the nonzero proper subspaces of the vector space $\mathbb F_2^n$. Hence $\cS(G)$ is the building of proper nonzero subspaces of $\mathbb F_2^n$. By Theorem~\ref{solomontits}, its order complex is homotopy equivalent to a wedge of spheres of dimension $n-2$; equivalently, since $\cL(G)$ is a geometric lattice, one has $\widetilde H_{n-2}(\cS(G))\neq 0$
for $n\geq 2$. Therefore $\cS(G)$ is not contractible when $n\geq 2$. If $n=1$, then $\cS(G)$ is empty. 
Therefore $\cS(G)$ is contractible if and only if $G$ is not elementary abelian. 
\end{proof}

The modularity assumption is actually stronger than needed. Nevertheless, in the present context it is useful to note that every nonelementary abelian modular $2$-group has contractible subgroup poset. We first point out that the contractibility of the proper subgroup poset of a finite $2$-group does not, by itself, characterize modularity.

\begin{remark}\label{rem:ordinary-contractibility-insufficient}\rm  Lemma \ref{crucialem} shows that ordinary contractibility only detects whether the Frattini subgroup is nontrivial. It does not distinguish modular and nonmodular $2$-groups. For instance, we have the following cases: \[ \begin{array}{c|c|c} P& \cL(P)\text{ is modular}& \cS(P)\text{ is contractible}\\ \hline \mathbb{Z}(2)^n,\ n\geq 2& \text{yes}& \text{no}\\ Q_8& \text{yes}& \text{yes}\\ D_8& \text{no}& \text{yes}. \end{array} \] Neither contractibility nor noncontractibility of $\cS(P)$ alone characterizes modularity. \end{remark} 

 We now isolate the local configuration which detects a dihedral section. 
 
 \begin{lemma}\label{lem:two-involutions-contractibility} Let $M$ be a finite $2$-group  possessing two distinct subgroups $A$ and $B$  of order $2$ generating $ K=\langle A,B\rangle.$ Then precisely one of the following alternatives occurs: \begin{enumerate} 
 \item[{\rm (i).}] $K\cong \mathbb{Z}(2)^2$, in which case  $\cS(K)$ is not contractible,  $\widetilde H_0(\cS(K);\mathbb Z)\cong\mathbb Z^2,$   $\mu(1,K)=2$; 
 \item [\rm (ii).] $K \cong D_8$ in which case  $\cS(K)$ is contractible,  $\widetilde H_j(\cS(K);\mathbb Z)=0$ for all $j$,  $\mu(1,K)=0.$ \end{enumerate} Consequently,  $\cS(\langle A,B\rangle)$ is not contractible if, and only if,  $\mathrm{comm}(A,B)=1$. \end{lemma}

 \begin{proof} Write  $A=\langle a\rangle \cong \mathbb{Z}(2)$ and  $B=\langle b\rangle \cong \mathbb{Z}(2)$ for suitable $a$ and $b$ distinct involutions and write $x=ab$ for their product. Since $axa=a(ab)a=ba=(ab)^{-1}=x^{-1},$ the group $\langle a,b\rangle=\langle x,a\rangle$ is dihedral, that is, either isomorphic to $\mathbb{Z}(2)^2$, or to $D_8$. If $a$ and $b$ commute, then $x^2=1$, and  $K \cong \mathbb{Z}(2)^2.$  Its proper nontrivial subgroup poset consists of its three subgroups of order $2$, with no comparability relations. Hence its order complex is a discrete space with three points. Therefore  $\widetilde H_0(\cS(K);\mathbb Z)\cong\mathbb Z^2$ and $\mu(1,K)=\widetilde\chi(\cS(K))=2.$ If $a$ and $b$ do not commute, then $x$ has order at least $4$. Since $M$ is a finite $2$-group, the order of $x$ is a power of $2$, and $K \cong D_8$. In particular, $K$ is not elementary abelian. Lemma \ref{crucialem}  implies $\cS(K)$  contractible. Its reduced homology and reduced Euler characteristic vanish, so  $\mu(1,K)=0.$ \end{proof}

 We can now prove Theorem \ref{thm:sectional-contractibility-modular}.

 \begin{proof}[Proof of Theorem \ref{thm:sectional-contractibility-modular}] From Iwasawa Theorem \cite{iwasawa} on  modular finite $2$-groups as by Theorem~\ref{grandeiwasawa},  $\cL(G)$  is modular if and only if $G$ is dihedral-free. Let $M$ be a section of $G$, and let $A,B$ be distinct subgroups of order $2$ in $M$. By Lemma~\ref{lem:two-involutions-contractibility},  $\cS(\langle A,B\rangle)$ is not contractible if and only if $\mathrm{comm}(A,B)=1.$ Thus (ii) is equivalent to the assertion that every pair of involutions commutes in every section of $G$. Suppose that some section $M$ contains two noncommuting involutions. They generate a dihedral $2$-group of order at least $8$. Every dihedral $2$-group of order at least $8$ contains a subgroup isomorphic to $D_8$. It follows that $G$ has a section isomorphic to $D_8$. Conversely, if $G$ has a section isomorphic to $D_8$, then that section contains two noncommuting involutions. Therefore (ii) fails. We have consequently proved that (i) is equivalent to (ii).
 
 Now assume (iii).  Then  $|A/N|=|B/N|=2$ and   the interval  $[N, \langle A, B \rangle]$ is isomorphic to the subgroup lattice of the quotient group $\langle A, B\rangle /N$, see \cite{S}. In particular,  $(N,\langle A, B\rangle) \simeq \cS(\langle A, B\rangle/N)$.   Then $\langle A, B\rangle /N$ is generated by the two distinct  subgroups $A/N$ and $B/N$ of order two. Lemma  \ref{lem:two-involutions-contractibility} now shows that conditions  (ii), (iii) and (iv) are  equivalent to the assertion that $ \langle A, B \rangle/N\cong \mathbb{Z}(2)^2$. In fact, this is equivalent to say that $\mathrm{comm}(A/N,B/N)=1$. \end{proof}

 We observe here that it is not sufficient to impose the condition only on pairs of involutions of $G$ itself. For example, $Q_{16}$ has a unique subgroup of order $2$. Hence any condition involving pairs of distinct subgroups of order two in $Q_{16}$ is vacuous. On the other hand,  $Q_{16}/Z(Q_{16})\cong D_8.$ Thus $Q_{16}$ has a dihedral section and is nonmodular. This explains why Theorem~\ref{thm:sectional-contractibility-modular} must involve all sections of $G$.

 \begin{remark}\rm Considering $\Z(4)$, we know that ${\mathcal S}(G)$ modular does not imply that any open interval is noncontractible, in fact $(1,\Z(4))=\{\Z(2)\}$.
Also the converse fails, that is,  we cannot claim that ${\mathcal S}(G)$ is modular, just knowing that any open interval is noncontractible. However Bj\"orner Complementation Theorem \cite{bjorner81, bjorner92} shows that with complementation we can have implications: in fact  in our case $\mathcal{S}(G)$ is a finite modular lattice and the order complex $|\Delta((H,K))|$ of the open interval $(H,K)$ is noncontractible if and only if the closed interval $[H,K]$ is complemented. This property applies for arbitrary modular finite lattices, see \cite{bjorner81, bjorner92, kratzer}. In particular, we have for a finite group $G$ that
every open interval  $(H,K)$ is noncontractible if and only if $G$ is elementary abelian. \end{remark}

The preceding results may be summarized as follows, at least for finite 2-groups: 
 \begin{enumerate}
     \item $\cS(G)$ is  contractible if and only if $\mathrm{Frat}(G)$ is nontrivial.
     \item $G$ is modular if and only if $\cS(\langle A,B\rangle)$  is noncontractible for every pair of distinct subgroups $A,B$ of order two  in every section of $G$. 
 \end{enumerate}
 In other words, we observe that while the ordinary contractibility is a Frattini invariant, the sectional contractibility detects modularity and heavily depends on the notions in Definition \ref{qfree}.

 \section{The sectional contractibility graph and  its properties} \label{s3}
 
 The Sectional Contractibility Criterion of Theorem~\ref{thm:sectional-contractibility-modular} can be encoded in an opportune family of graphs, which we call sectional contractibility graphs. In this section we show a modularity characterization in terms of magnitude homology and show some properties of 
 sectional contractibility graphs.
 
 \begin{definitionn}\label{nicenicegraph} Let $G$ be a finite $2$-group. The \emph{sectional contractibility graph} of $G$, denoted by $\Gamma_{\mathrm{ct}}(G),$ is the graph whose vertex set is  \[V(\Gamma_{\mathrm{ct}}(G)) = \{\,A\leq G\mid |A|=2\,\}\ . \] 
 Two distinct vertices $A$ and $B$ are joined by an edge if and only if $\cS(\langle A,B\rangle)$ is not contractible. 
 \end{definitionn} 

For a finite group $G$, one can also consider the  
\textit{commuting graph} $\Gamma_c(G)$ -- see, e.g.~\cite{alireza06, mohammed15, rowley21}. This is the  simple undirected graph  with vertex set
$V(\Gamma_c(G)) =G\setminus Z(G)$ and elements $x,y\in G$  joined if and only if $xy=yx$. We may consider all the involutions 
\[
\mathrm{Inv}(G)=\{g\in G \mid g^2=1 \ \mbox{and}  \ g\neq1\}
\]
and study the commuting graph $\Gamma_\mathrm{inv}(G)$ on $\mathrm{Inv}(G)$ (instead of $G\setminus Z(G)$), restricting to $\mathrm{Inv}(G)$ as set of vertices and connecting two involutions $u,v$ of $G$ if and only if they commute. The graph $\Gamma_\mathrm{inv}(G)$ is called \textit{commuting involution graph} of $G$ and also previously considered, see e.g.~\cite{rowley21}.  We have the following result establishing a connection between the sectional contractibility graph and the commuting involution graph;

 \begin{prop}\label{prop:contractibility-commuting-graph} Let $G$ be a finite $2$-group with $A,B$ subgroups of $G$ of order two. Then  $\{A,B\}\in E(\Gamma_{\mathrm{ct}}(G))$ if and only if  $\mathrm{comm}(A,B)=1.$ Thus $\Gamma_{\mathrm{ct}}(G)=\Gamma_{\mathrm{inv}}(G)$. \end{prop} 
 \begin{proof} This  follows from Lemma~\ref{lem:two-involutions-contractibility}. In fact, Two $A$ and $B$ are two distinct involutions and  either $\langle A,B\rangle\cong \Z(2)^2$ when
 they commute, or $\langle A,B\rangle\cong D_8$ when they do not. In the first case, the proper subgroup poset consists of three isolated points, so it is noncontractible. In the second case, the generated group has nontrivial Frattini subgroup, and its proper subgroup poset is contractible. Thus,  in the present context, adjacency means exactly commutativity.  \end{proof}

 \begin{cor} \label{cor:graph-characterization-modular} Let $G$ be a finite $2$-group. Then  $G$ is nonmodular  if and only if  some section~$M$ of $G$ contains two vertices $A,B$ for which  $\cS(\langle A,B\rangle)$  is contractible. \end{cor} 
 \begin{proof} By Proposition~\ref{prop:contractibility-commuting-graph}, the graph $\Gamma_{\mathrm{ct}}(M)$ is complete if and only if every two involutions in~$M$ commute. The conclusion follows from Theorem~\ref{thm:sectional-contractibility-modular}. \end{proof}

\begin{figure}
    \centering
    \includegraphics[width=1\linewidth]{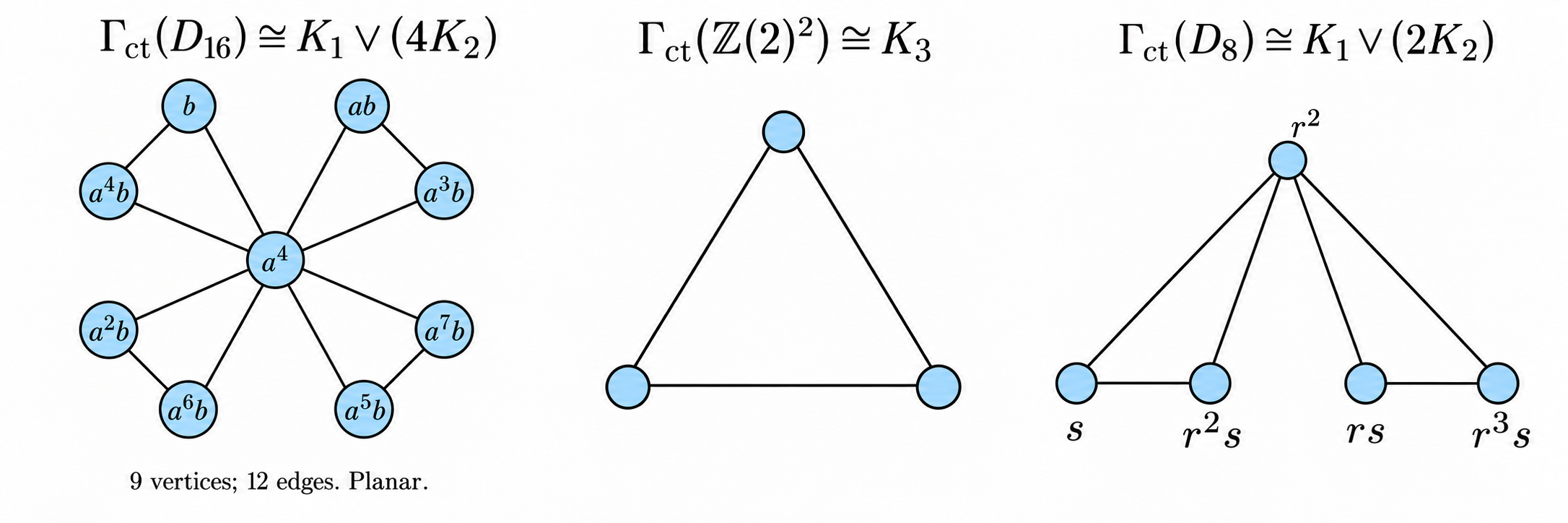}
    \caption{Some examples of sectional contractibility graphs which are planar.}
    \label{examplesofplanarity}
\end{figure}

  For a finite $2$-group $G$, define the following numbers \[ a(G) \coloneqq \bigl|\{\,A\leq G\mid |A|=2\,\}\bigr|\] and 
  \[v(G) \coloneqq \bigl|\{\,V\leq G\mid V\cong \mathbb{Z}(2)^2\,\}\bigr|\ . \]

  \begin{lemma}\label{lem:number-edges-klein-four} For every finite $2$-group $G$,  we have \[ \bigl|E(\Gamma_{\mathrm{ct}}(G))\bigr|=3v(G)\ . \] \end{lemma} 
  
  \begin{proof} Every   elementary abelian 2-subgroup $ \mathbb{Z}(2)^2$ of rank two contains precisely three subgroups of order $2$, and hence determines three unordered pairs of distinct subgroups of order two. Conversely, two distinct  subgroups $A,B$ of order two are adjacent in $\Gamma_{\mathrm{ct}}(G)$ if and only if they commute. In this case $ \langle A,B\rangle\cong \mathbb{Z}(2)^2$,  and this subgroup is uniquely determined by the pair $\{A,B\}$. Therefore each elementary abelian 2-subgroup of rank two  of this type contributes exactly three edges, and no edge is counted twice. \end{proof}

  \begin{cor}\label{cor:numerical-characterization-modular} Let $G$ be a finite $2$-group. Then $\cL(G)$ is modular if and only if \[ 3v(M)=\binom{a(M)}{2} \] for every section $M$ of $G$. \end{cor} 
  \begin{proof} By Lemma~\ref{lem:number-edges-klein-four},  $|E(\Gamma_{\mathrm{ct}}(M))|=3v(M).$ The graph $\Gamma_{\mathrm{ct}}(M)$ has $a(M)$ vertices and is complete if and only if \[ |E(\Gamma_{\mathrm{ct}}(M))| =\binom{a(M)}{2}\ . \] The conclusion follows from Corollary~\ref{cor:graph-characterization-modular}. \end{proof}

 We now reformulate the preceding criterion in terms  of magnitude homology. Recall the decomposition of magnitude homology for posets given by Theorem~\ref{thm:decompositionMH}.

 \begin{definitionn}[Bottom Supported Magnitude Homology]\label{bottomsupported} Let $Q$ be a finite $2$-group. We define the bottom supported part of $\MH_{2,2}(G_{\cL(Q)})$ by \[ \MH^{\,1}_{2,2}(G_{\cL(Q)}) \coloneqq \bigoplus_{\substack{K\leq Q\\ |K|=4}} \widetilde H_0((1,K);\mathbb Z)\ . \]
 \end{definitionn} 
 
 Note that, by the interval decomposition of magnitude homology in  Theorem~\ref{thm:decompositionMH}, $\MH^{\,1}_{2,2}(G_{\cL(Q)})$ is a direct summand of $ \MH_{2,2}(G_{\cL(Q)})$.

 \begin{lemma}\label{prop:bottom-magnitude-rank} Let $Q$ be a finite $2$-group. Then \[ \MH^{\,1}_{2,2}(G_{\cL(Q)})\cong\mathbb Z^{\,2v(Q)}. \] 
 \end{lemma}

 \begin{proof} Every subgroup of order $4$ is isomorphic either to $\mathbb{Z}(4)$ or to $\mathbb{Z}(2)^2$. If $K\cong \mathbb{Z}(4)$, then the open interval $(1,K)$ consists of the unique subgroup of $K$ of order $2$. Hence  $\widetilde H_0((1,K);\mathbb Z)=0.$  If $K\cong \mathbb{Z}(2)^2$, then $(1,K)$ consists of three incomparable subgroups of order $2$. Therefore  $\widetilde H_0((1,K);\mathbb Z)\cong\mathbb Z^2.$ Since there are $v(Q)$ subgroups of type $\mathbb{Z}(2)^2$ in $Q$, the result follows. \end{proof} 
 
We can now provide  the proof of Theorem~\ref{cor:magnitude-characterization-modular}.

 \begin{proof}[Proof of Theorem~\ref{cor:magnitude-characterization-modular}] By Lemma \ref{prop:bottom-magnitude-rank},  $\operatorname{rank}_{\mathbb Z}\MH^{\,1}_{2,2}(G_{\mathcal{L}(M)})=2v(M).$ On the other hand, \[ 3v(M)=\binom{a(M)}{2} \] in Corollary \ref{cor:numerical-characterization-modular} is equivalent to \[ 2v(M)=\frac{a(M)(a(M)-1)}{3}\ . \] The conclusion  follows from Corollary~\ref{cor:numerical-characterization-modular} again. \end{proof}

 The intervalwise formulation is essential here. The total group $\MH_{2,2}(G_{\cL(M)})$ may contain contributions from many intervals whose lower endpoint is not the trivial subgroup. The summand  $\MH^{\,1}_{2,2}(G_{\mathcal{L}(M)})$ isolates precisely the intervals $(1,K)$ with $|K|=4$, and therefore records exactly the subgroups of type $\mathbb{Z}(2)^2$ and the commuting pairs of involutions in $M$.

We know from \cite{alireza06, mohammed15} that the commuting graph associated to a finite \(2\)-group $P$ is \textit{planar} (i.e.~it can be embedded in the usual plane in such a way that no edges cross each other) if and only if, either
$P \cong D_8$, or $P \cong Q_8$, or $|P|=16$ and $|Z(P)|=4$. We conclude the section with the proof of the planarity result for sectional contractibility graphs:

\begin{proof}[Proof of Theorem \ref{planarity}] If \(G\) contains \(\Z(2)^3\), its seven involutions commute pairwise. They give a complete graph \(K_7\) in \(\Gamma_{\mathrm{ct}}(G)\), so the graph is nonplanar.
Conversely, suppose that \(G\neq1\) contains no \(\Z(2)^3\). Choose a central involution $z\in Z(G).$ It is adjacent to every other involution. The involutions other than \(z\) are partitioned into pairs $\{a,az\}.$ Each pair consists of commuting involutions and, together with \(z\), forms a triangle. Suppose that distinct involutions \(a,b\neq z\) commute. Then
$\langle z,a,b\rangle$ is elementary abelian. Since it cannot have order \(8\), it has order at most~\(4\), forcing $b=az$. There are therefore no edges between different pairs. The graph is exactly a collection of triangles sharing the vertex \(z\), and is planar. \end{proof}

\begin{figure}
    \centering
    \includegraphics[width=1\linewidth]{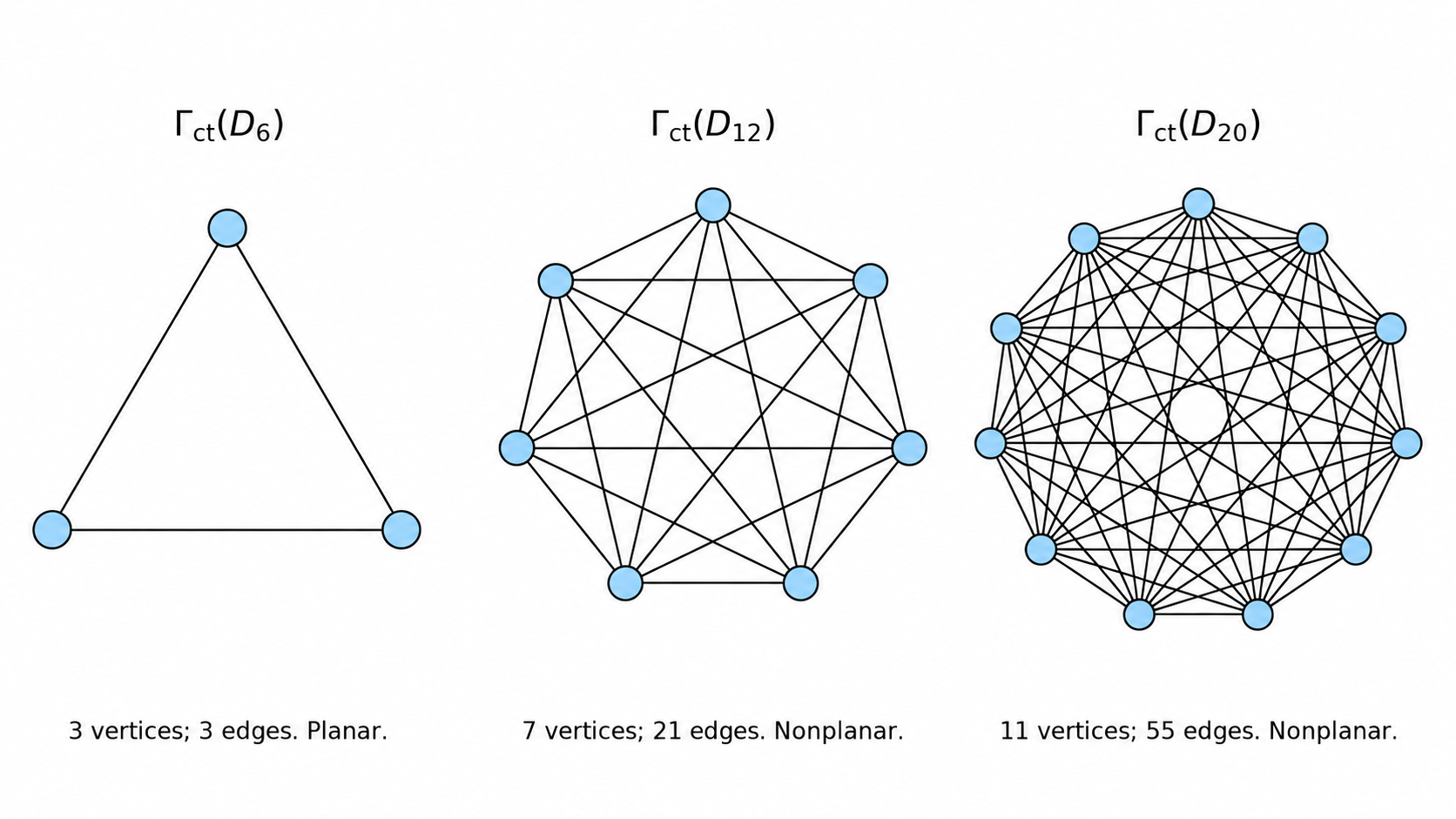}
    \caption{Further examples of sectional contractibility graphs, arising from dihedral groups which are not  2-groups.}
    \label{figmoreusefuldihedrals}
\end{figure}

Unfortunately, Theorem \ref{planarity}  shows that planarity does not characterize modularity: \(D_8\) is nonmodular but $\Gamma_{\mathrm{ct}}(D_8)$ is planar, whereas \(\Z(2)^3\) is modular but $\Gamma_{\mathrm{ct}}(\Z(2)^2)$ is nonplanar.

\begin{problem} \rm
The spectrum of the Laplacian matrix and of the adjacency matrix of the sectional contractibility graph have not been studied yet, to the best of our knowledge. In particular, recent works involve the spectral graph theory of certain graphs which are attached to Pauli groups, see \cite{BDDR}, and this allows to recognize bipartite graphs just looking at symmetric spectrum, or other combinatorial properties. On the other hand, there is no  much  literature on  the spectrum of the sectional contractibility graph of Pauli groups, and, more generally, of large families of finite groups.   
\end{problem}

\section{A focus on almost extraspecial 2-groups}\label{s4}

We have seen in Theorem~\ref{thm:bjorner} that if $G$ is supersolvable, then  $\cL(G)$ is Cohen-Macaulay. Hence, its magnitude homology is diagonal, that is,    it vanishes everywhere off the main diagonal. In this section we compute the magnitude homology for various families of groups. We start  with the cyclic groups. 

\begin{example} \rm \

\begin{enumerate}
    \item Let  $G= \mathbb{Z}(p)$ be a cyclic group of prime order. In such case, it is easy to compute the magnitude homology of $\cL(G)$:
\[
\MH_{0,0}(G_{\cL(\mathbb{Z}(p))})=\Z^2 \text{ and } \MH_{1,1}(G_{\cL(\mathbb{Z}(p))})=\Z
\]
and all others magnitude homology groups are 0. Note that this is a trivial example of supsersolvable groups with diagonal magnitude homology. Further, it is such that  $\cS(G)$ is Cohen-Macalauy. 
\item Consider the group $G =\mathbb{Z}(p^n)$ for $n \ge 1$. Then $\cL(\mathbb{Z}(p^n))$ is a  line graph, hence the magnitude homology is still concentrated in bidegree $(0,0)$ and $(1,1)$ with \[\MH_{0,0}(G_{\cL(\mathbb{Z}(p^n))})= \Z^{n+1}  \ \mbox{and} \ \MH_{1,1}(G_{\cL(\mathbb{Z}(p^n))})= \Z^{n} \ .\] 
\item If  $G=\mathbb{Z}(pq)=\mathbb{Z}(p) \times \mathbb{Z}(q)$, with $p,q$ primes such that $\gcd(p,q)=1$, then
\[
\MH_{0,0}(G_{\cL(\mathbb{Z}(pq))})=\Z^4\ , \quad \MH_{1,1}(G_{\cL(\mathbb{Z}(pq)})=\Z^4 \text{ ,  and } \MH_{2,2}(G_{\cL(\mathbb{Z}(pq))})=\Z \ .
\]
The Frattini subgroup is trivial and in fact $\cS(G)$ is not contractible. This gives homology in $\MH_{\ell,\ell}(G_{\cL(\mathbb{Z}(pq))})$, where $\ell$ is 2.
\end{enumerate}
   \end{example}

More generally, using Theorem~\ref{thm:bjorner}, the graphs associated to subgroups lattices of  supersolvable groups are diagonal. Furthermore, the  ranks of magnitude homology can be given in terms of the M\"obius function.

We now focus on extraspecial $p$-groups. They are interesting because, in view of Theorem~\ref{thm:bjorner}, they are:
\begin{itemize}
    \item supersolvable -- in fact, all finite $p$-groups are nilpotent, and all finitely generated nilpotent groups are supersolvable;
    \item  of length of  maximal chain bigger than $2$;  
    \item of nontrivial Frattini subgroup;
    \item  of  contractible poset of proper nontrivial subgroups. 
\end{itemize}
We start with the dihedral group:

\begin{example}\label{magnitudehomologyofd8} \rm
    Consider the extraspecial 2-group $G=D_8$. Then
    \[
    \MH_{k,\ell}(G_{\cL(D_8)})=\begin{cases}
        \Z^{10} & \text{if } (k,\ell)=(0,0)\ \;\\
        \Z^{15}& \text{if } (k,\ell)=(1,1)\ \;\\
        \Z^{6}& \text{if } (k,\ell)=(2,2)\ \;\\
    \end{cases}
    \]
    Here the first is the number of subgroups, the second the number of edges, the third the number of relative connected components of the following Hasse diagram -- see also Figure~\ref{fig:d8-lattice}. 
We have
$$\cL(D_8)=\{ \{1\}, \langle s \rangle, \langle s^2\rangle, \langle r \rangle, \langle sr \rangle, \langle s^2r \rangle, \langle s^3r \rangle, \{1,s^2,r,s^2r\}, \{1,s^2,sr,s^3r\}, D_8 \}\ ,$$ with the inclusion relation. Denoting $B=\langle b \rangle$, $Z(D_8)=\langle r^2 \rangle$, $M_1=\{1,s^2,r,s^2r\}$, $M_2=\{1,s^2,sr,s^3r\}$, we have $$\mathcal{N}(D_8)=\{D_8, \{1\}, B, Z(D_8), M_1, M_2\},$$ as sublattice of all normal subgroups of $D_8$. Notice that $H=\langle s^2r \rangle$ and $K=\langle r\rangle$ are contained in $M_1$, while $U=\langle sr\rangle $ and $V= \langle s^3r \rangle$ in $M_2$. 
\end{example}

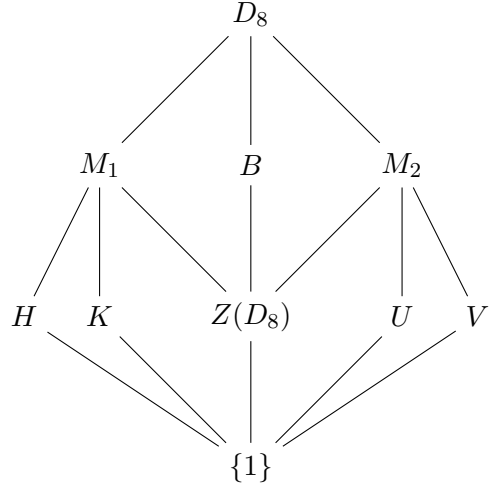
\begin{figure}[ht]
\centering
\begin{tikzpicture}
  \node (max) at (0,4) {$D_8$};
  \node (a) at (-2,2) {$M_1$};
  \node (b) at (0,2) {$B$};
  \node (c) at (2,2) {$M_2$};
  \node (d) at (-2,0) {$K$};
  \node (g) at (-3,0) {$H$};
  \node (e) at (0,0) {$Z(D_8)$};
  \node (f) at (2,0) {$U$};
  \node (h) at (3,0) {$V$};
  \node (min) at (0,-2) {$\{1\}$};
  \node (Fig) at (4,-2) {};
  \draw (min) -- (g) -- (a) -- (d) -- (min) -- (e)--
  (a) -- (max) -- (b) -- (e) -- (c) -- (f) -- (min) --(h)--(c) -- (max);
 \end{tikzpicture}
 \caption{A draft of the Hasse diagram of the subgroups lattice $\cL(D_8)$.}
\label{fig:d8-lattice}
 \end{figure}

Something more general can be noted for arbitrary extraspecial $p$-groups for any prime $p$.

\begin{prop}
Let $G$ be an extraspecial $p$-group.    Then, the magnitude homology of $\cL(G)$ is nontrivial only in bidegrees $(0,0),(1,1), (2,2)$. 
\end{prop}

\begin{proof}
    Extraspecial $p$-groups are supersolvable and the maximal chain length in the associated subgroups lattice is at most $2$. The statement follows.
\end{proof}

The Pauli group $P$ on $16$ elements has been recently studied in \cite{BBR, BDDR, OR}.  It fits Definition \ref{defextrasp} since $P$ is almost extraspecial 2-group. 
In fact, \[P=\{\pm I,\pm iI,\pm X,\pm iX,\pm Y,\pm iY,\pm Z,\pm iZ\} \tag{\dag}\]
is the one-qubit Pauli group, where $i \in \mathbb{C}$, $I$ is the identity matrix 2-by-2 and the following matrices are in $\mathrm{Gl}_2(\mathbb{C})$
\[
X=\begin{pmatrix}0 & 1\\ 1 & 0
\end{pmatrix},
\qquad
Y=\begin{pmatrix}0&-i\\ i & 0\end{pmatrix},
\qquad
Z=\begin{pmatrix}1&0\\ 0&-1\end{pmatrix}.
\]
Then $| P|=16$, and $ P$ admits the presentation
\[
 P=
\langle iI,X,Z \mid (iI)^4=I,\ (iI)^{-1} X^{-1}(iI) X= (iI)^{-1} Z^{-1}(iI) Z=I, \ X^2=Z^2=I,\ ZX=(iI)^2XZ\rangle,
\]
where $\omega=iI$, $Z( P)=\langle iI\rangle\cong \mathbb{Z}(4)$, $\mathrm{comm}(P,P)=\langle -I\rangle\cong \mathbb{Z}(2)$ and a direct inspection of the maximal subgroups shows that $\mathrm{Frat}(P)=\langle -I\rangle.$ Hence $\mathrm{Frat}( P)\neq 1$, so $\cS( P)$ is contractible by \cite[Proposition~1.8]{kratzer}.
The subgroup lattice $\cL(P)$ has:
\begin{itemize}
\item $7$ subgroups of order $2$, namely;
\[ \langle X\rangle,  \langle Y\rangle, \langle Z\rangle, \langle - X\rangle, \langle -Y\rangle, \langle -Z\rangle, \langle -I\rangle; \] 
\item $7$ subgroups of order $4$, namely
\[
V_X=\langle -I,X\rangle,
V_Y=\langle -I,Y\rangle,
V_Z=\langle -I,Z\rangle,
\mathbb{Z}(4)=\langle iI\rangle,
C_X=\langle iX\rangle,
C_Y=\langle iY\rangle,
C_Z=\langle iZ\rangle;
\]
\item $7$ maximal subgroups of order $8$, namely
\[
M_X=\langle iI,X\rangle,
M_Y=\langle iI,Y\rangle,
M_Z=\langle iI,Z\rangle,
\]
which are all isomorphic to $\mathbb{Z}(4) \times \mathbb{Z}(2)$;
\[
D_{XY}=\langle X,Y\rangle,\qquad
D_{XZ}=\langle X,Z\rangle,\qquad
D_{YZ}=\langle Y,Z\rangle,
\]
which are all isomorphic to $D_8$; and
\[
Q=\langle iX,iY\rangle=\langle iX,iZ\rangle=\langle iY,iZ\rangle\cong Q_8;
\]
\item the trivial subgroup and the whole group.
\end{itemize}

A draft of the Hasse diagram of $\cL(P)$ is given in Figure~\ref{fig:pauli16-lattice}. It may be useful as a test case for the computation of the magnitude homology of subgroup lattices beyond the dihedral example. First note that the subgroups $A=\langle X \rangle$ and $B=\langle Y \rangle$ and do not commute, and  $\cS(\langle X,Y\rangle)$  is contractible. Thus the vertices $A$ and $B$ are not adjacent in  $\Gamma_{\mathrm{ct}}(P). $ In particular, $\Gamma_{\mathrm{ct}}(P)$ is not complete, and Theorem~\ref{thm:sectional-contractibility-modular} correctly detects that $P$ is not modular, even though its full proper subgroup poset $\cS(P)$ is contractible.  For $T\in\{X,Y,Z\}$, define
\[
A_T^+=\langle T\rangle,
\qquad
A_T^-=\langle -T\rangle,\qquad
V_T=\langle -I,T\rangle,
\qquad
C_T=\langle iT\rangle,\qquad
C=\langle iI\rangle.
\]
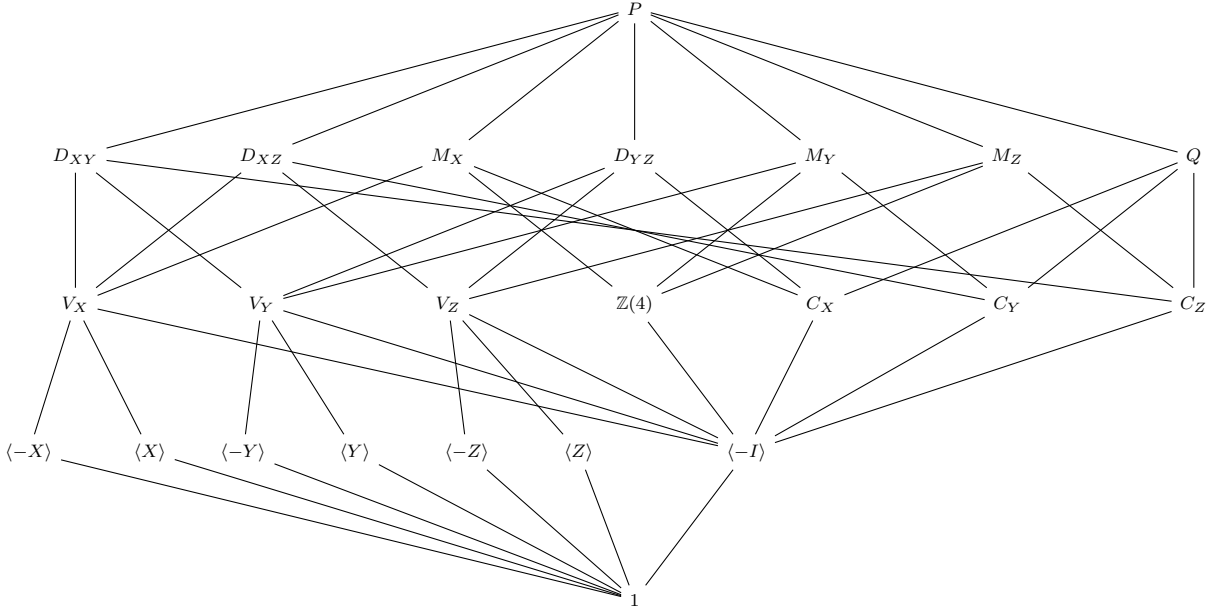
\begin{figure}[ht]
\centering
\begin{tikzpicture}[x=1.45cm,y=1.15cm,scale=0.85,transform shape,
every node/.style={font=\scriptsize}]

\node (P) at (0,8) {$P$};

\node (Dxy) at (-6,6) {$D_{XY}$};
\node (Dxz) at (-4,6) {$D_{XZ}$};
\node (Mx)  at (-2,6) {$M_X$};
\node (Dyz) at (0,6) {$D_{YZ}$};
\node (My)  at (2,6) {$M_Y$};
\node (Mz)  at (4,6) {$M_Z$};
\node (Q)   at (6,6) {$Q$};

\node (Vx) at (-6,4) {$V_X$};
\node (Vy) at (-4,4) {$V_Y$};
\node (Vz) at (-2,4) {$V_Z$};
\node (C4) at (0,4) {$\mathbb{Z}(4)$};
\node (Cx) at (2,4) {$C_X$};
\node (Cy) at (4,4) {$C_Y$};
\node (Cz) at (6,4) {$C_Z$};

\node (mX) at (-6.5,2) {$\langle -X\rangle$};
\node (X)  at (-5.2,2) {$\langle X\rangle$};
\node (mY) at (-4.2,2) {$\langle -Y\rangle$};
\node (Y)  at (-3.0,2) {$\langle Y\rangle$};
\node (mZ) at (-1.8,2) {$\langle -Z\rangle$};
\node (Z)  at (-0.6,2) {$\langle Z\rangle$};
\node (mI) at (1.2,2) {$\langle -I\rangle$};

\node (one) at (0,0) {$1$};

\draw (Dxy)--(P);
\draw (Dxz)--(P);
\draw (Mx)--(P);
\draw (Dyz)--(P);
\draw (My)--(P);
\draw (Mz)--(P);
\draw (Q)--(P);

\draw (Vx)--(Dxy);
\draw (Vx)--(Dxz);
\draw (Vx)--(Mx);

\draw (Vy)--(Dxy);
\draw (Vy)--(Dyz);
\draw (Vy)--(My);

\draw (Vz)--(Dxz);
\draw (Vz)--(Dyz);
\draw (Vz)--(Mz);

\draw (C4)--(Mx);
\draw (C4)--(My);
\draw (C4)--(Mz);

\draw (Cx)--(Mx);
\draw (Cx)--(Dyz);
\draw (Cx)--(Q);

\draw (Cy)--(Dxz);
\draw (Cy)--(My);
\draw (Cy)--(Q);

\draw (Cz)--(Dxy);
\draw (Cz)--(Mz);
\draw (Cz)--(Q);

\draw (mI)--(Vx);
\draw (mI)--(Vy);
\draw (mI)--(Vz);
\draw (mI)--(C4);
\draw (mI)--(Cx);
\draw (mI)--(Cy);
\draw (mI)--(Cz);

\draw (mX)--(Vx);
\draw (X)--(Vx);

\draw (mY)--(Vy);
\draw (Y)--(Vy);

\draw (mZ)--(Vz);
\draw (Z)--(Vz);

\draw (one)--(mX);
\draw (one)--(X);
\draw (one)--(mY);
\draw (one)--(Y);
\draw (one)--(mZ);
\draw (one)--(Z);
\draw (one)--(mI);

\end{tikzpicture}
\caption{A draft of the Hasse diagram of the subgroup lattice $\cL(P)$ of the Pauli group on $16$ elements.}
\label{fig:pauli16-lattice}
\end{figure}
\begin{prop}\label{prop:magnitude-Pauli-16}
Consider the Pauli group $G=P$ in $(\dag)$. Then the magnitude homology  of
$G_{\cL(P)}$ is given by
\[
\MH_{k,\ell}\bigl(G_{\cL(P)}\bigr)
\cong
\begin{cases}
\Z^{23},
    & (k,\ell)=(0,0),\\[1mm]
\Z^{48},
    & (k,\ell)=(1,1),\\[1mm]
\Z^{34},
    & (k,\ell)=(2,2),\\[1mm]
\Z^{8},
    & (k,\ell)=(3,3),\\[1mm]
0,
    & \text{otherwise}.
\end{cases}
\]
In particular, 
$G_{\cL(P)}$ is diagonal.
\end{prop}

\begin{proof}
The subgroup lattice $\cL(P)$ has rank $4$. Its subgroups,
grouped according to rank, are:
\[
\begin{array}{c|c|c}
\text{rank}&\text{order}&\text{number of subgroups}\\
\hline
0&1&1\\
1&2&7\\
2&4&7\\
3&8&7\\
4&16&1
\end{array}
\]
More explicitly, the seven subgroups of order $2$ are
$\mathrm{Frat}(P),$ $A_X^+,A_X^-,$ $A_Y^+,A_Y^-,$
$A_Z^+,A_Z^-.$
The seven subgroups of order $4$ are $V_X,V_Y,V_Z,C,C_X,C_Y,C_Z,$
and the seven subgroups of order $8$ are
$D_{XY},D_{XZ},D_{YZ},M_X,M_Y,M_Z,Q. $
Therefore $\cL(P)$ has
\[
1+7+7+7+1=23
\]
vertices, and hence
\[
\MH_{0,0}\bigl(G_{\cL(P)}\bigr)
\cong
\Z^{23}.
\]
Since $P$ is a finite $2$-group, it is supersolvable.
Consequently, its subgroup lattice is lexicographically shellable and
Cohen--Macaulay by Theorem~\ref{thm:bjorner}. 
Hence the graph $G_{\cL(P)}$ is diagonal. It
remains to compute the diagonal ranks.

\smallskip
\noindent
\emph{Length $1$.} Every cover relation $H\lessdot K$ contributes one copy of
\[
\widetilde H_{-1}(\varnothing;\Z)\cong\Z\ .
\]
There are seven covers from rank $0$ to rank $1$.

From rank $1$ to rank $2$, the subgroup $E$ is contained in all seven
subgroups of order $4$. Moreover, $A_T^+<V_T$ and $A_T^-<V_T$ 
for each $T\in\{X,Y,Z\}$, and these are the only further cover
relations between ranks $1$ and $2$. Thus there are
$
7+6=13
$
covers between these two ranks. The subgroups of order $4$ contained in each maximal subgroup are
listed in the following table:
\[
\begin{array}{c|c}
\text{maximal subgroup}
&
\text{subgroups of order $4$}
\\
\hline
D_{XY}
&
V_X,\ V_Y,\ C_Z
\\
D_{XZ}
&
V_X,\ V_Z,\ C_Y
\\
D_{YZ}
&
V_Y,\ V_Z,\ C_X
\\
M_X
&
V_X,\ C,\ C_X
\\
M_Y
&
V_Y,\ C,\ C_Y
\\
M_Z
&
V_Z,\ C,\ C_Z
\\
Q
&
C_X,\ C_Y,\ C_Z.
\end{array}
\]
It follows that there are
$
7\cdot3=21
$
covers from rank $2$ to rank $3$. Finally, there are seven covers
from the maximal subgroups to $P$. Thus the total number
of edges in the directed Hasse graph is
$
7+13+21+7=48
$.
Consequently,
\[
\MH_{1,1}\bigl(G_{\cL(P)}\bigr)
\cong
\Z^{48}.
\]

\smallskip
\noindent
\emph{Length $2$.} An open interval of length $2$ is a discrete set. If it consists of
$m$ points, then
$
\widetilde H_0((H,K);\Z)\cong\Z^{m-1}
$.
The nonzero contributions are divided into three families. First, for each $T\in\{X,Y,Z\}$, the interval $(1,V_T)$ consists of the three  subgroups of order two
$\mathrm{Frat}(P), \ A_T^+,\ A_T^-$. Hence $$\widetilde H_0((1,V_T);\Z)\cong\Z^2. $$
There are three intervals of this form. Second, let $L$ be any maximal subgroup of $P$. Every
 subgroup of order four of $L$ contains $\mathrm{Frat}(P)$, and the preceding table shows
that there are exactly three such subgroups. Therefore $(\mathrm{Frat}(P),L)$
is a discrete set of three points, and $$
\widetilde H_0((\mathrm{Frat}(P),L);\Z)\cong\Z^2.$$ There are seven intervals of this form. Third, the same table shows that every subgroup $R$ of order $4$ is
contained in exactly three maximal subgroups of $P$. Therefore $(R,P)$ is again a discrete set of three points, and
\[
\widetilde H_0((R,P);\Z)\cong\Z^2.
\]
There are seven intervals of this form. All the remaining intervals of length $2$ contain exactly one point and consequently have trivial reduced homology. It follows that
\[
\MH_{2,2}\bigl(G_{\cL(P)}\bigr) \cong
(\Z^2)^{3}
\oplus
(\Z^2)^{7}
\oplus
(\Z^2)^{7}
\cong
\Z^{34}.
\]

\smallskip
\noindent
\emph{Length $3$.} There are two kinds of intervals of length $3$. Let first $L$ be a maximal subgroup of~$P$. Then
$(1,L)=\cS(L).$ For $L\cong D_8$, or
$L\cong \mathbb{Z}(4)\times \mathbb{Z}(2)$, or
$L\cong Q_8,$
we have $\mathrm{Frat}(L)= \mathrm{Frat}(P)\neq1.$
Hence $(1,L)$ is contractible and gives no contribution to magnitude
homology. Consider now a subgroup $A$ of $P$ of order two. Since
$\mathrm{Frat}(P)$ is central and normal, we may consider
$(\mathrm{Frat}(P),P)
\cong
\cS(P/\mathrm{Frat}(P))$, where we known (by construction of $P$) that the quotient $P/\mathrm{Frat}(P)\cong(\mathbb{Z}(2))^3$. By the Solomon--Tits Theorem (see Theorem \ref{solomontits}),
\[
\cS\bigl((\mathbb{Z}(2))^3\bigr)
\simeq
\bigvee_{2^{\binom32}}S^1
=
\bigvee_8S^1.
\]
Consequently,
\[
\widetilde H_1((\mathrm{Frat}(P),P);\Z)
\cong
\Z^8.
\]
Suppose instead that $A=A_T^+$ or
$A=A_T^-,$ where $T\in\{X,Y,Z\}$. The maximal subgroups containing $A$ are
precisely the three maximal subgroups containing $V_T$. Their
intersection is $V_T$. Thus $\mathrm{Frat}(A,P)=V_T>A.$
By Lemma \ref{kratzerlemma}, the interval $(A,P)$ is contractible.
It follows that the only nonzero interval contribution of length $3$
comes from $(\mathrm{Frat}(P),P)$. Therefore
\[
\MH_{3,3}\bigl(G_{\cL(P)}\bigr)
\cong
\Z^8.
\]

\smallskip
\noindent
\emph{Length $4$.} The only interval of length $4$ is $(1,P)=\cS(P).$
Since $\mathrm{Frat}(P)\neq1$,this interval is contractible. Hence
\[
\MH_{4,4}\bigl(G_{\cL(P)}\bigr)=0.
\]
There are no intervals of length greater than $4$. This completes the
proof.
\end{proof}

\begin{figure}
    \centering
    \includegraphics[width=0.7\linewidth]{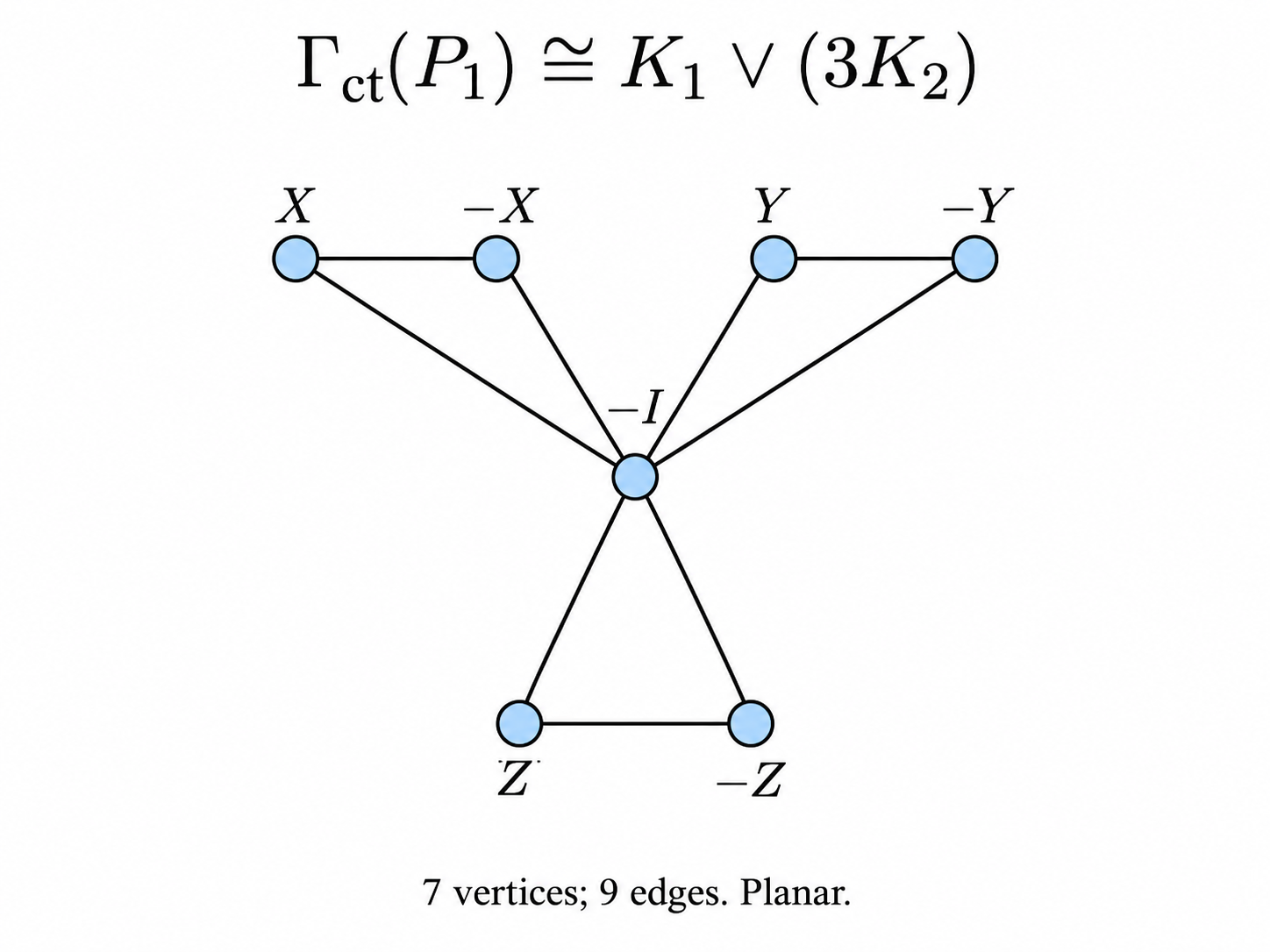}
    \caption{Sectional contractibility graph of the Pauli group $P_1$ and its planarity.}
    \label{pone}
\end{figure}

\begin{remark} \rm
The nonzero M\"obius values appearing in the preceding computation
are
\[
\mu(H,K)=
\begin{cases}
-1,
&
K\text{ covers }H,
\\[1mm]
2,
&
(H,K)=(1,V_T),
\quad T\in\{X,Y,Z\},
\\[1mm]
2,
&
(H,K)=(E,L),
\quad |L|=8,
\\[1mm]
2,
&
(H,K)=(R,P),
\quad |R|=4,
\\[1mm]
-8,
&
(H,K)=(E,P).
\end{cases}
\]
Every other interval of positive length has M\"obius value zero.
Accordingly,  diagonal ranks are
\[
23,
\qquad
48,
\qquad
(3+7+7)\cdot2=34,
\qquad
8.
\]
It is also useful to note that the bigraded rank polynomial of the magnitude homology is 
\[
\sum_{k,\ell\geq0}
\operatorname{rank}_{\Z}
\MH_{k,\ell}\bigl(G_{\cL(P)}\bigr)
u^kq^\ell
=
23+48uq+34u^2q^2+8u^3q^3.
\]
Since the magnitude homology is diagonal, its graded Euler
characteristic is
\[
23-48q+34q^2-8q^3.
\]
\end{remark}

On the basis of the notions which we have seen in Definition \ref{defextrasp}, it is good to point out that  $P$ is almost extraspecial and not extraspecial. In fact,
 the preceding computations are very rich and cannot be deduced from a statement
concerning extraspecial $p$-groups only. Compare Example \ref{magnitudehomologyofd8} and Proposition \ref{prop:magnitude-Pauli-16}. We now extend the preceding computation for the Pauli group of order
$16$ to the full family of almost extraspecial $2$-groups.

\begin{remark} \rm
The full $r$-qubit Pauli group, including the scalar matrices
$\{\pm I,\pm iI\}$, is an almost extraspecial $2$-group of order $2^{2r+2}$ of the form \[P_r=E^+_{2^{2r+1}} \circ \mathbb{Z}(2^2). \tag {\dag \dag}\] In particular, the Pauli group of order $16$ considered previously is
the case $r=1$, since it can be seen that $P\cong D_8\circ \mathbb{Z}(4)$, see \cite{BBR, BDDR, OR}. \end{remark}

Fix an almost extraspecial $2$-group $P_r$, and note that
\[
\mathrm{Frat}(P_r)=\mathrm{comm}(P_r,P_r) \cong \mathbb{Z}(2).
\]
Since $\mathrm{Frat}(P_r)$, the Burnside Basis Theorem~\cite{berk1}  (or by direct inspection on the construction of almost extraspecial 2-groups in Definition \ref{defextrasp}) we have that
\[
P_r/\mathrm{Frat}(P_r)\cong \mathbb{Z}(2)^{2r+1} \ .
\]
For integers $0\leq b\leq a$, let
\[
\qbinom{a}{b}
=
\prod_{j=0}^{b-1}
\frac{2^{a-j}-1}{2^{b-j}-1}
\]
denote the Gaussian binomial coefficient over $\mathbb F_2$ and use
the convention that $\qbinom{a}{b}=0$ when
$b<0$ or $b>a$, while $
\qbinom{a}{0}=1$ when $b=0$. The following notion can be found in \cite{berk1}, but is manipulated ad hoc here. This Gaussian binomial coefficients helps to count the total number of linear subspaces
\[
T_m=\sum_{j=0}^{m}\qbinom{m}{j}
\]
of $\mathbb F_2^m$. The quotient $P_r/\mathrm{Frat}(P_r)$ is a vector space over $\mathbb F_2$ in the present situation and every element of $P_r/\mathrm{Frat}(P_r)$ turns out to be a square, or product of squares.  To see this, choose the nontrivial element $z\in \mathrm{Frat}(P_r)$. Since $P_r/\mathrm{Frat}(P_r)$ is
elementary abelian 2-group of finite rank, one has $x^2\in \mathrm{Frat}(P_r)$ for every $x\in P_r$. Define $ q\colon P_r/\mathrm{Frat}(P_r) \to \mathbb F_2$
by $q(x\mathrm{Frat}(P_r))=0$ if $x^2=1$, and $q(x\mathrm{Frat}(P_r))=1 $ if $x^2=z$. Then the commutator $
x^{-1}y^{-1}xy$ is $z^m$ for some $m\in\{0,1\}$. The vector space $P_r/\mathrm{Frat}(P_r)$ has dimension $2r+1$ and Witt index $r$. A subspace $W\leq P_r/\mathrm{Frat}(P_r)$ is called \emph{totally singular} if $q(w)=0$
for every $w\in W$. 

\begin{lemma}\label{lem:elementary-abelian-singular}
Let $E$ be a subgroup of $P_r$ containing $\mathrm{Frat}(P_r)$. Then $E$ is elementary
abelian if and only if $E/\mathrm{Frat}(P_r)$ is a totally singular subspace of $P_r/\mathrm{Frat}(P_r)$.
\end{lemma}

\begin{proof}
Suppose that $E$ is elementary abelian. Every element of $E$ has
order at most $2$, so $q(x\mathrm{Frat}(P_r))=0$ for every $x\in E$. Moreover, all elements of $E$ commute. Hence $E/\mathrm{Frat}(P_r)$ is totally singular. Conversely, suppose that $E/\mathrm{Frat}(P_r)$ is totally singular. Then $x^2=1$ for every $x\in E$, and $xy=yx$ for every $x,y\in E$. Therefore $E$ is elementary abelian.
\end{proof}

For $1\leq t\leq r+1$, define
\[
e_t(r)
 =
\qbinom{r}{t-1}
\prod_{j=0}^{t-2}\bigl(2^{r-j}+1\bigr) \ \mbox{and put}  \ 
e_t(r)=0
\ \text{if }t<1\text{ or }t>r+1.
\]
As usual, an empty product is equal to $1$, so $e_1(r)=1.$

\begin{prop}\label{prop:number-elementary-containing-N}
The number of elementary abelian subgroups $E$ of $P_r$ containing
$\mathrm{Frat}(P_r)$ and of $|E|=2^t$
is precisely $e_t(r)$.
\end{prop}

\begin{proof}
By Lemma~\ref{lem:elementary-abelian-singular}, these subgroups
correspond bijectively to totally singular subspaces of dimension
$t-1$ in the quadratic space $V_r$. The standard enumeration of
totally singular subspaces in a quadratic space of dimension $2r+1$
and Witt index $r$ gives $e_t(r)$.
\end{proof}

The number of subgroups of finite $p$-group has been largely investigated by Berkovich and others, see \cite{berk1}.

\begin{lemma}\label{prop:number-subgroups-almost-extraspecial}
The total number of subgroups
$s_r=|\cL(P_r)|$ is given by 
\[
s_r
 =
1+T_{2r+1}
+\sum_{i=1}^{r}2^i e_{i+1}(r).
\]
\end{lemma}

\begin{proof} See \cite{berk1} or use inspection on the structure of $P_r$.
\end{proof}

The following observation is useful beyond the almost extraspecial
case.

\begin{lemma}\label{prop:intervals-finite-2-groups}
Let $G$ be a finite $2$-group with $H<K\leq G$ and $\ell=\log_2|K:H|.$
Then one (and only one) of the following alternatives occurs:
\begin{enumerate}
    \item[{\rm (i).}] If $\mathrm{Frat}(K)\nleq H$, then the open interval $(H,K)$ is
    contractible;

    \item[{\rm (ii).}] If $\mathrm{Frat}(K)\leq H$, then $H$ is normal in $ K$,
        $K/H\cong \mathbb{Z}(2)^\ell$ and    $    (H,K)\simeq\cS\bigl(\mathbb{Z}(2)^\ell\bigr)$.
        Moreover in this second case, we have
    \[
    |\Delta(H,K)|
    \simeq
    \bigvee_{2^{\binom{\ell}{2}}}S^{\ell-2} \ \ \mbox{and} \ \
\widetilde H_j((H,K);\mathbb Z)
\cong
\begin{cases}
\mathbb Z^{\,2^{\binom{\ell}{2}}},
   &j=\ell-2,\\
0,&j\neq\ell-2.
\end{cases}
\]
\end{enumerate}
\end{lemma}

\begin{proof} (i). Suppose  $\mathrm{Frat}(K)\nleq H$ and define the map
$c\colon(H,K)\longrightarrow(H,K)$ from any $H<X<K$ to
$c(X)= X \mathrm{Frat}(K)$ again in the interval $(H,K)$.
It can be easily checked that $c$ is well defined, since 
$X\mathrm{Frat}(K)=K$ would  imply
$X=K$ by the Frattini properties of $K$, which is impossible. Thus
$ H<X\mathrm{Frat}(K)<K,$ since $\mathrm{Frat}(K)\nleq H$. The map $c$ is an ascending closure operator between intervals, since $X\leq c(X)$ and $c(c(X))=c(X)$. Its image has the least element
$H\mathrm{Frat}(K)$. Hence its image is contractible, and so also
 $(H,K)$ is contractible, since we noted that $c$ is a closure operator.

(ii). Suppose now $\mathrm{Frat}(K)\leq H.$ Since the commutator subgroup $[K,K]$ is contained in $\mathrm{Frat}(K)$, and so in $H$, we conclude that the subgroup $H$ is normal in $K$. Moreover,
$K^2=\{k^2 \mid k \in K\}\leq\mathrm{Frat}(K)\leq H.$ Therefore $K/H$ is abelian and every element of $K/H$ has order at
most $2$. Hence $K/H\cong \mathbb{Z}(2)^\ell.$ Looking at the corresponding interval, we find that
$[H,K]\simeq\cL(K/H),$ and consequently
\[
(H,K)\simeq\cS(K/H)
\simeq\cS\bigl(\mathbb{Z}(2)^\ell\bigr).
\]
The Solomon--Tits Theorem (see Theorem \ref{solomontits}) implies
$|\Delta(H,K)|$ as requested, and the homology formula follows. Note that
for $\ell=1$, the interval is empty. We use the usual augmented homology convention
$\widetilde H_{-1}(\varnothing;\mathbb Z)\cong\mathbb Z,$
which agrees with
$2^{\binom12}=1.$ The result follows.
\end{proof}

The following result can be seen as a consequence of various works, such as~\cite{quillen78, Shareshian01, sharu12, sharu16}, but we provide a different argument and this may be instructive. 

\begin{cor}\label{cor:magnitude-finite-2-groups-diagonal}
Let $G$ be a finite $2$-group. Then 
$G_{\cL(G)}$ is diagonal and its magnitude homology torsion-free. 
\end{cor}

\begin{proof} We consider $H$ and $K$ as per Lemma \ref{prop:intervals-finite-2-groups} and recall  that the magnitude homology is
\[
\MH_{k,\ell}\bigl(G_{\cL(G)}\bigr)
\cong
\bigoplus_{\substack{H<K\\
r(K)-r(H)=\ell}}
\widetilde H_{k-2}((H,K);\mathbb Z).
\]
Since the subgroup lattice of a finite $2$-group is ranked by
\[
r(H)=\log_2|H|,
\]
the condition $r(K)-r(H)=\ell$ is equivalent to
\[
|K:H|=2^\ell.
\]
The result follows from
Lemma ~\ref{prop:intervals-finite-2-groups}.
\end{proof}

\begin{figure}
    \centering
    \includegraphics[width=1 \linewidth]{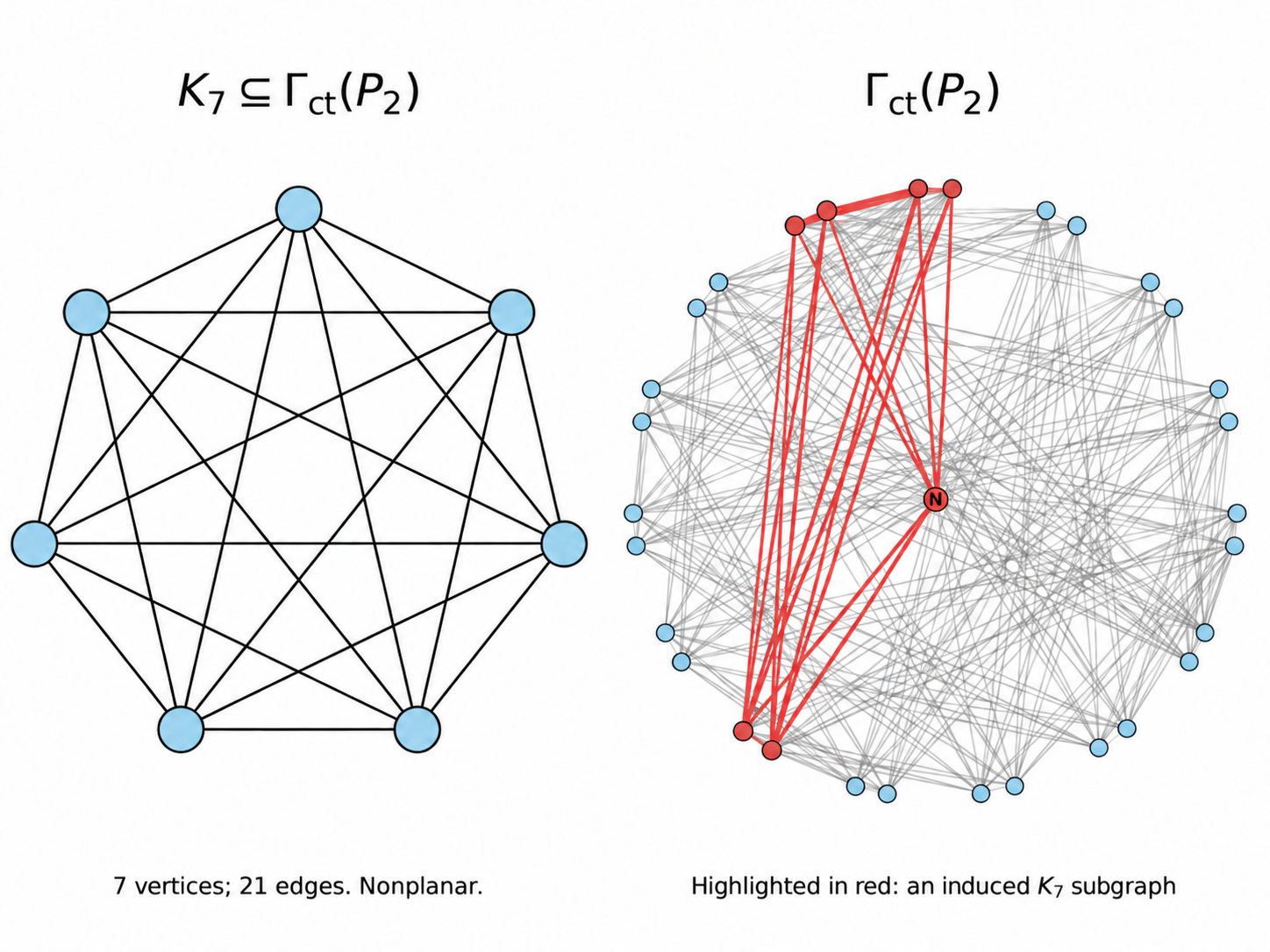}
    \caption{Sectional contractibility graph of the Pauli group $P_2$ and its nonplanarity.}
    \label{ptwononplanar}
\end{figure}

Again some further notions from \cite{berk1} can be recalled here, but we specialize it to the present case of almost extraspecial 2-groups.

\begin{definitionn} \rm
For $1\leq\ell\leq2r+1$ and $P_r$ almost extraspecial 2-group, define
\[
\nu_\ell(r)
 =
\left|
\left\{
(H,K)\ \middle|\
H\unlhd K\leq P_r,\quad
K/H\cong \mathbb{Z}(2)^\ell
\right\}
\right|
\]
as the number of elementary abelian sections of
rank $\ell$, counted together with their lower and upper endpoints.
\end{definitionn}

\begin{lemma}\label{prop:number-elementary-sections}
For $1\leq\ell\leq2r+1$ and $P_r$ almost extraspecial 2-group, we have
\[
\nu_\ell(r)
=
\qbinom{2r+1}{\ell}T_{2r+1-\ell} 
+e_\ell(r)+2^\ell e_{\ell+1}(r) 
+\sum_{m=\ell+1}^{r}
   2^m\qbinom{m}{\ell}e_{m+1}(r)
+\sum_{m=\ell+1}^{r+1}
   2^{m-\ell}
   \qbinom{m-1}{\ell-1}e_m(r),
\]
where an empty sum is understood to be zero.
\end{lemma}

\begin{proof}
We divide the  sections $K/H\cong \mathbb{Z}(2)^\ell$ into four mutually exclusive families.

\smallskip
\noindent
\emph{First family: $\mathrm{Frat}(P_r)\leq H$.} In this case,
$\mathrm{Frat}(P_r) \leq H<K$,
and both $H/\mathrm{Frat}(P_r)$ and $K/\mathrm{Frat}(P_r)$ are subspaces of
$P_r/\mathrm{Frat}(P_r)\cong\mathbb F_2^{2r+1}$. Write
\[
a=\dim_{\mathbb F_2}(H/\mathrm{Frat}(P_r)).
\]
Then
\[
\dim_{\mathbb F_2}(K/\mathrm{Frat}(P_r))=a+\ell.
\]
The number of such pairs is
\[
\sum_{a=0}^{2r+1-\ell}
\qbinom{2r+1}{a}
\qbinom{2r+1-a}{\ell}.
\]
Using the Gaussian identity
\[
\qbinom{2r+1}{a}
\qbinom{2r+1-a}{\ell}
=
\qbinom{2r+1}{\ell}
\qbinom{2r+1-\ell}{a},
\]
this sum becomes
\[
\qbinom{2r+1}{\ell}
\sum_{a=0}^{2r+1-\ell}
\qbinom{2r+1-\ell}{a}
=
\qbinom{2r+1}{\ell}T_{2r+1-\ell},
\]
where we clearly identify the first summand in $\nu_\ell(r)$.

\smallskip
\noindent
\emph{Second family: $H=1$.} In this case, $K$ itself is elementary abelian of order $2^\ell$. There are $e_\ell(r)$ such subgroups containing $\mathrm{Frat}(P_r)$. Suppose instead that $ \mathrm{Frat}(P_r) \nleq K$. Then $K\mathrm{Frat}(P_r)$ is elementary abelian of order $2^{\ell+1}$ and contains $\mathrm{Frat}(P_r)$. Conversely, $K$ is a complement of $\mathrm{Frat}(P_r)$ in $K\mathrm{Frat}(P_r)$. There are $2^\ell$ such complements for each possible subgroup $K\mathrm{Frat}(P_r)$. Hence the number of elementary abelian subgroups $K$ of order $2^\ell$ not containing
$\mathrm{Frat}(P_r)$ is $2^\ell e_{\ell+1}(r).$
The contribution of this family is therefore
\[
e_\ell(r)+2^\ell e_{\ell+1}(r),
\]
where we identify the second summand in the requested expression of $\nu_\ell(r)$.

\smallskip
\noindent
\emph{Third family: $1<H<K$ and $ \mathrm{Frat}(P_r)\nleq K$.} Since the commutator subgroup
$[K,K]\leq N\cap K=1$ and $K^2\leq N\cap K=1,$  we have that $K$ is elementary abelian and so $|K|=2^m.$ Since $H\neq1$ and $|K:H|=2^\ell$, one has $m\geq\ell+1.$ The group $K \mathrm{Frat}(P_r)$ is elementary abelian of order $2^{m+1}$ and contains
$\mathrm{Frat}(P_r)$. There are $e_{m+1}(r)$
choices for $K\mathrm{Frat}(P_r)$, and each such group has
$2^m$ complements to $\mathrm{Frat}(P_r)$. Once $K$ is fixed, the number of subgroups $H$ of
codimension $\ell$ in $K$ is $\qbinom{m}{\ell}.$ Since $m\leq r$, the contribution is
\[
\sum_{m=\ell+1}^{r}
2^m\qbinom{m}{\ell}e_{m+1}(r),
\]
where we identify one more of the requested summands.

\smallskip
\noindent
\emph{Fourth family: $1<H<K$, $\mathrm{Frat}(P_r) \leq K$, and $ \mathrm{Frat}(P_r) \nleq H$.} Since $K/H$ is elementary abelian, $\mathrm{Frat}(K)\leq H$ and also $K/\mathrm{Frat}(P_r)$ is elementary abelian, so $\mathrm{Frat}(K)\leq \mathrm{Frat}(P_r).$ Because $ H\cap \mathrm{Frat}(P_r)=1,$
it follows that $\mathrm{Frat}(K)=1.$ Thus $K$ is elementary abelian, say $|K|=2^m.$ There are
$e_m(r)$ possible elementary abelian subgroups $K$ of this order containing
$\mathrm{Frat}(P_r)$. The subgroup $H$ has dimension $m-\ell$ and must not contain
$\mathrm{Frat}(P_r)$. The number of $(m-\ell)$-dimensional subspaces of $K$ not
containing the fixed $1$-dimensional subspace $\mathrm{Frat}(P_r)$ is
\[
\qbinom{m}{m-\ell}
-
\qbinom{m-1}{m-\ell-1}.
\]
By the Gaussian-Pascal  identity, this is
\[
2^{m-\ell}
\qbinom{m-1}{\ell-1}.
\]
Since $H\neq1$, one has $m\geq\ell+1$. Moreover, the maximal possible
dimension of an elementary abelian subgroup containing $\mathrm{Frat}(P_r)$ is
$r+1$. Thus the contribution is
\[
\sum_{m=\ell+1}^{r+1}
2^{m-\ell}
\qbinom{m-1}{\ell-1}e_m(r) 
\]
appears. Adding the four contributions  the formula is proved.
\end{proof}

We get to the proof of the last main result of the present paper.

\begin{proof}[Proof of Theorem \ref{thm:magnitude-almost-extraspecial}]
The group $\MH_{0,0}\bigl(G_{\cL(P_r)}\bigr)$ is free abelian on the vertices of the Hasse diagram. Hence
\[
\MH_{0,0}\bigl(G_{\cL(P_r)}\bigr)
\cong\mathbb Z^{\,|\cL(P_r)|}
=\mathbb Z^{\,s_r}
\]
by Lemma \ref{prop:number-subgroups-almost-extraspecial}. For $\ell\geq1$, the interval decomposition gives
\[
\MH_{k,\ell}\bigl(G_{\cL(P_r)}\bigr)
\cong
\bigoplus_{\substack{H<K\\|K:H|=2^\ell}}
\widetilde H_{k-2}((H,K);\mathbb Z).
\]
By Lemma \ref{prop:intervals-finite-2-groups}, an interval gives
a nonzero contribution if and only if
\[
\mathrm{Frat}(K)\leq H,
\]
or equivalently, $H\unlhd K$
and
$K/H\cong \mathbb{Z}(2)^\ell$. There are exactly $\nu_\ell(r)$ such pairs. For each such pair, the Solomon--Tits Theorem (see Theorem \ref{solomontits}) gives
\[
\widetilde H_{k-2}((H,K);\mathbb Z)
\cong
\begin{cases}
\mathbb Z^{\,2^{\binom{\ell}{2}}},
   &k=\ell,\\
0,&k\neq\ell.
\end{cases}
\]
Therefore $\MH_{\ell,\ell}\bigl(G_{\cL(P_r)}\bigr)
\cong
\mathbb Z^{
  \,2^{\binom{\ell}{2}}\nu_\ell(r)}$ for $ 1\leq\ell\leq2r+1.$
The only interval of length $2r+2$ is
$(1,G)=\cS(G).$ Since $\mathrm{Frat}(G) \neq1$, this interval is contractible. Hence
\[
\MH_{2r+2,\,2r+2}
\bigl(G_{\cL(P_r)}\bigr)=0.
\]
This completes the proof.
\end{proof}

\begin{cor}\label{cor:top-magnitude-almost-extraspecial}
Let $G=P_r$ be an almost extraspecial 2-group. Then
Theorem~\ref{thm:magnitude-almost-extraspecial} implies that
\[
\nu_{2r+1}(r)=1 \ \ \mbox{and} \ \ 
\MH_{2r+1,\,2r+1}
\bigl(G_{\cL(P_r)}\bigr)
\cong
\mathbb Z^{\,2^{r(2r+1)}}.
\]
Moreover,
$\MH_{2r+2,\,2r+2}
\bigl(G_{\cL(P_r)}\bigr)=0.$
\end{cor}

\begin{proof}
The unique elementary abelian section of rank $2r+1$ is
$G/\mathrm{Frat}(G)\cong \mathbb{Z}(2)^{2r+1}.$
Thus $\nu_{2r+1}(r)=1.$ Since
\[
\binom{2r+1}{2}=r(2r+1),
\]
the formula follows from
Theorem~\ref{thm:magnitude-almost-extraspecial}.
\end{proof}

Although  $\cL(P_r)$ has rank
\[
\log_2|P_r|=2r+2,
\]
the top nonzero magnitude-homology group (see Corollary \ref{cor:top-magnitude-almost-extraspecial}) occurs in bidegree
\[
(2r+1,2r+1).
\]
The disappearance of the diagonal term in degree $2r+2$ reflects the
contractibility of the full proper subgroup poset
$\cS(P_r)=(1,P_r).$ On the other hand, the nonzero group in degree $2r+1$ is produced by
the Frattini quotient
$P_r/\mathrm{Frat}(P_r)\cong \mathbb{Z}(2)^{2r+1}.$

\begin{remark} \rm According to Theorem \ref{thm:magnitude-almost-extraspecial}, 
the bigraded rank polynomial of the magnitude homology of an almost extraspecial 2-group $G=P_r$ can be written as
\[
\sum_{k,\ell\geq0}
\operatorname{rank}_{\mathbb Z}
\MH_{k,\ell}\bigl(G_{\cL(P_r)}\bigr)
u^kq^\ell
=
s_r+
\sum_{\ell=1}^{2r+1}
2^{\binom{\ell}{2}}
\nu_\ell(r)(uq)^\ell.
\]
\end{remark}

\begin{example} \rm The case $r=1$ of an almost extraspecial 2-group gives back what we noted in case of the Pauli group on one qubit. Now $P=P_1\cong D_8\circ \mathbb{Z}(4)$ as per $(\dag)$.  Here
\[
d=3,
\qquad
e_1(1)=1,
\qquad
e_2(1)=3,
\]
and
\[
T_3
 =
\qbinom30+\qbinom31+\qbinom32+\qbinom33
 =
1+7+7+1
 =
16.
\]
Thus
\[
s_1
 =
1+T_3+2e_2(1)
 =
1+16+6
 =
23.
\]
The formula of
Lemma \ref{prop:number-elementary-sections} gives
\[
\nu_1(1)=48,
\qquad
\nu_2(1)=17,
\qquad
\nu_3(1)=1.
\]
Consequently,
\[
\MH_{k,\ell}
\bigl(P_{\cL(P)}\bigr)
\cong
\begin{cases}
\mathbb Z^{23},
   &(k,\ell)=(0,0),\\
\mathbb Z^{48},
   &(k,\ell)=(1,1),\\
\mathbb Z^{34},
   &(k,\ell)=(2,2),\\
\mathbb Z^{8},
   &(k,\ell)=(3,3),\\
0,
   &\text{otherwise}.
\end{cases}
\]
Indeed,
$2^{\binom22}\nu_2(1)=2\cdot17=34$
and $2^{\binom32}\nu_3(1)=8.$ Thus the general formula recovers the preceding direct computation for
$P$.
\end{example}

\begin{example} \rm The Pauli group on 2-qubits (here as $r=2$ and $P_2$ in $(\dag \dag)$) was studied in \cite{BDDR} under different perspectives. It turns out to be useful also in mathematical physics, as discussed in \cite{babaru, BBR}. Now
$P_2\cong D_8\circ D_8\circ \mathbb{Z}(4)$ has $|P_2|=64$ and
$P_2/\mathrm{Frat}(P_2)\cong \mathbb{Z}(2)^5.$ We have
$e_1(2)=1$,
$e_2(2)=15$,
$e_3(2)=15,$
and
\[
T_5
 =
1+31+155+155+31+1
 =
374.
\]

\begin{figure}[ht]
\centering
\begin{tikzpicture}[
    family/.style={
        draw,
        rounded corners=3pt,
        align=center,
        minimum width=32mm,
        minimum height=10mm,
        font=\small
    },
    singleton/.style={
        draw,
        rounded corners=3pt,
        align=center,
        minimum width=26mm,
        minimum height=9mm,
        font=\small
    },
    cover/.style={
        -{Latex[length=2mm]},
        semithick
    },
    level/.style={
        font=\scriptsize\itshape
    }
]

\node[singleton] (L0) at (0,0)
{$1$};

\node[singleton] (N) at (-2.8,1.6)
{$\mathrm{Frat}(P_2)$};

\node[family] (A2) at (2.8,1.6)
{$30$ elementary abelian subgroups\\[-1mm]
{\scriptsize order $2$, not containing $\mathrm{Frat}(P_2)$}};

\node[family] (EA4) at (-6.0,3.4)
{$15$ elementary abelian subgroups\\[-1mm]
{\scriptsize of order 4, containing $\mathrm{Frat}(P_2)$}};

\node[family] (C4) at (0,3.4)
{$16$  cyclic subgroups \\[-1mm]
{\scriptsize  of order $4$, containing $\mathrm{Frat}(P_2)$}};

\node[family] (EA4n) at (6.0,3.4)
{$60$ elementary abelian subgroups\\[-1mm]
{\scriptsize of order 4, not containing $\mathrm{Frat}(P_2)$}};

\node[family] (EA8) at (-4.0,5.3)
{$15$ elementary abelian subgroups\\[-1mm]
{\scriptsize of order 8, containing $\mathrm{Frat}(P_2)$}};

\node[family] (NE8) at (3.3,5.3)
{$140$ nonelementary abelian subgroups\\[-1mm]
{\scriptsize of order $8$, containing $\mathrm{Frat}(P_2)$}};

\node[family] (NE16) at (0,7.2)
{$155$ subgroups\\[-1mm]
{\scriptsize order $16$, all containing $\mathrm{Frat}(P_2)$}};

\node[family] (NE32) at (0,8.9)
{$31$ maximal subgroups\\[-1mm]
{\scriptsize order $32$, all containing $\mathrm{Frat}(P_2)$}};

\node[singleton] (P2) at (0,10.6)
{$P_2$};

\draw[cover] (L0) -- (N);
\draw[cover] (L0) -- (A2);

\draw[cover] (N) -- (EA4);
\draw[cover] (N) -- (C4);

\draw[cover] (A2) -- (EA4n);

\draw[cover] (EA4) -- (EA8);
\draw[cover] (EA4) -- (NE8);

\draw[cover] (C4) -- (NE8);

\draw[cover] (EA4n) -- (EA8);
\draw[cover] (EA4n) -- (NE8);

\draw[cover] (EA8) -- (NE16);
\draw[cover] (NE8) -- (NE16);

\draw[cover] (NE16) -- (NE32);

\draw[cover] (NE32) -- (P2);

\end{tikzpicture}

\caption{This is a ``compressed Hasse diagram'' of
$\cL(P_2)$ in the sense that the subgroups of  $P_2$ are collected in families of subgroups which  contain (or not)  $\mathrm{Frat}(P_2)$. These families are then ordered by ascending Gaussian binomial coefficients $\qbinom{5}{0}$,  $\qbinom{5}{1}$,  $\qbinom{5}{2}$,  $\qbinom{5}{3}$,  $\qbinom{5}{4}$,  $\qbinom{5}{5}$ which contribute to form $T_5$.}
\label{fig:hasse-P2-almost-extraspecial}
\end{figure}
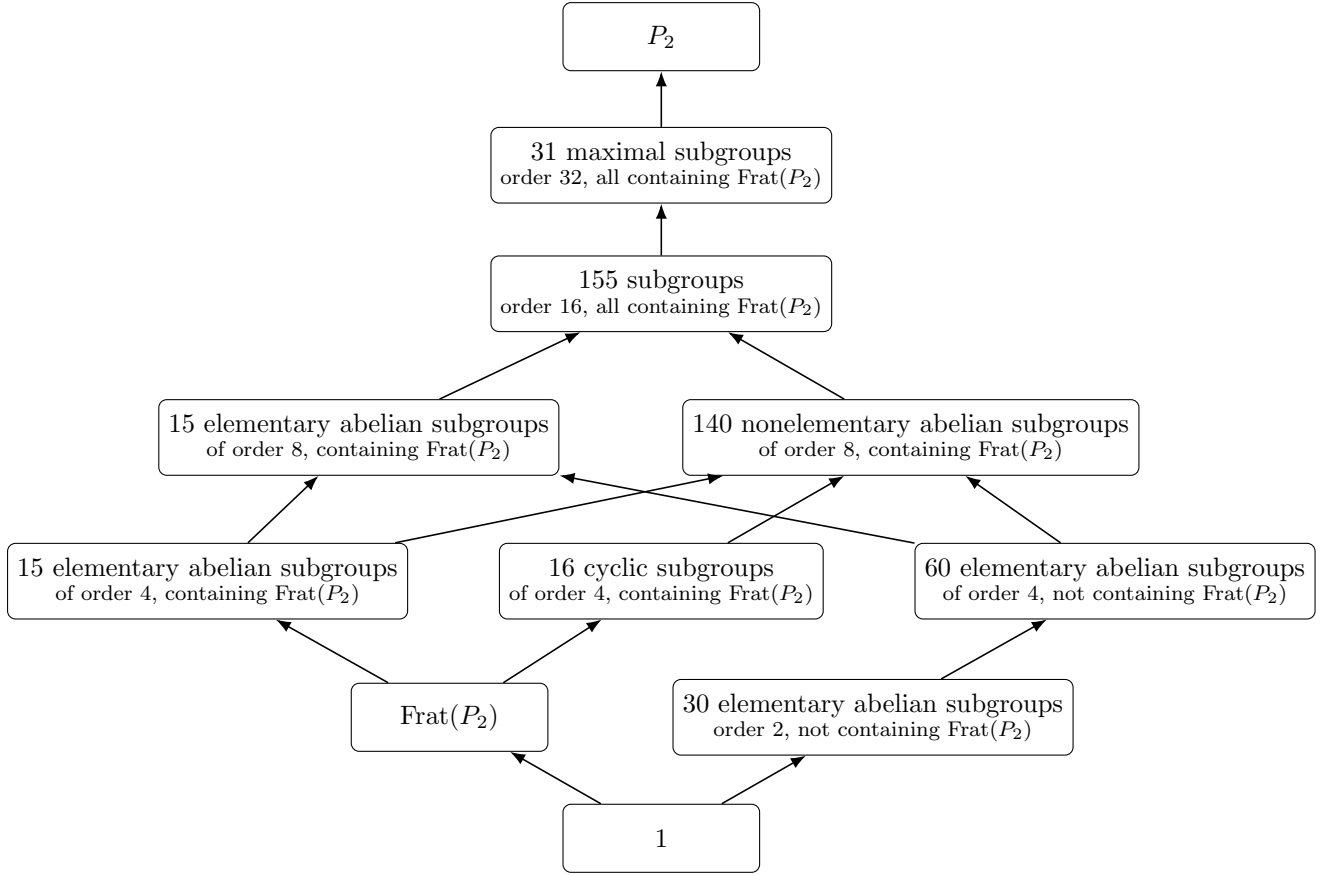

Therefore
\[
s_2
=
1+T_5+2e_2(2)+4e_3(2)
=
1+374+30+60
=
465.
\]
The numbers of elementary abelian sections are
\[
\nu_1(2)=2378, \ 
\nu_2(2)=2645, \ 
\nu_3(2)=790, \ 
\nu_4(2)=62, \ 
\nu_5(2)=1.
\]
It follows for $G=P_2$ that
\[
\MH_{k,\ell}
\bigl(G_{\cL(P_2)}\bigr)
\cong
\begin{cases}
\mathbb Z^{465},
   &(k,\ell)=(0,0),\\
\mathbb Z^{2378},
   &(k,\ell)=(1,1),\\
\mathbb Z^{5290},
   &(k,\ell)=(2,2),\\
\mathbb Z^{6320},
   &(k,\ell)=(3,3),\\
\mathbb Z^{3968},
   &(k,\ell)=(4,4),\\
\mathbb Z^{1024},
   &(k,\ell)=(5,5),\\
0,
   &\text{otherwise}.
\end{cases}
\]
Indeed, we have
$
2^{\binom22}\nu_2(2)
 =
2\cdot2645
 =
5290, \ 
2^{\binom32}\nu_3(2)
 =
8\cdot790
 =
6320, \ 
2^{\binom42}\nu_4(2)
 =
64\cdot62
 =
3968, \ 
2^{\binom52}\nu_5(2)
 =
1024.
$
\end{example}

\bibliographystyle{alpha}

\bibliography{biblio}
\end{document}